\documentclass[a4paper,10pt,oneside,headsepline]{scrartcl}

\usepackage{etoolbox}
\usepackage{ifdraft}

\usepackage[utf8]{inputenc}
\usepackage[T1]{fontenc}

\usepackage{lmodern}
\usepackage{fourier}

\usepackage{amssymb, amsfonts}

\usepackage{mathrsfs} 
\usepackage{dsfont}
\usepackage[usenames,dvipsnames]{xcolor}
\usepackage{lettrine}
\usepackage{url}

\usepackage{stmaryrd}

\usepackage[english]{babel}

\usepackage[draft=false]{scrlayer-scrpage}

\usepackage{wrapfig}
\usepackage{footmisc}
\usepackage{needspace}

\usepackage{amsmath}
\usepackage{nccmath}
\usepackage[thmmarks, amsmath, amsthm]{ntheorem}

\providebool{nobiblatex}
\boolfalse{nobiblatex}
\ifbool{nobiblatex}{%
}{%
  \usepackage[backend=biber,style=numeric-comp,giveninits,url=false,sortcites=true,maxbibnames=99]{biblatex}
  \addbibresource{bibliography.bib}
}

\usepackage[pdftex,final,
            pdfborder={0 0 0}, colorlinks=true,
            linkcolor=BrickRed, citecolor=ForestGreen, urlcolor=RoyalBlue]{hyperref}

\ifdraft{%
\usepackage[textsize=scriptsize,colorinlistoftodos]{todonotes}
\usepackage{lineno}
}{}

\usepackage[ruled]{algorithm2e}
\usepackage{booktabs}
\usepackage[inline]{enumitem}
\usepackage{float}
\usepackage{graphicx}
\usepackage{multicol}
\usepackage{scrtime}
\usepackage{tikz}
\usetikzlibrary{matrix,arrows,positioning}
\usepackage{adjustbox}

\ifdraft{\date{\today}}{\date{}}

\ifdraft{%
\linenumbers
\newcommand*\patchAmsMathEnvironmentForLineno[1]{%
  \expandafter\let\csname old#1\expandafter\endcsname\csname #1\endcsname
  \expandafter\let\csname oldend#1\expandafter\endcsname\csname end#1\endcsname
  \renewenvironment{#1}%
     {\linenomath\csname old#1\endcsname}%
     {\csname oldend#1\endcsname\endlinenomath}}%
\newcommand*\patchBothAmsMathEnvironmentsForLineno[1]{%
  \patchAmsMathEnvironmentForLineno{#1}%
  \patchAmsMathEnvironmentForLineno{#1*}}%
\AtBeginDocument{%
\patchBothAmsMathEnvironmentsForLineno{equation}%
\patchBothAmsMathEnvironmentsForLineno{align}%
\patchBothAmsMathEnvironmentsForLineno{flalign}%
\patchBothAmsMathEnvironmentsForLineno{alignat}%
\patchBothAmsMathEnvironmentsForLineno{gather}%
\patchBothAmsMathEnvironmentsForLineno{multline}%
}}

\KOMAoptions{DIV=10}

\clearmainofpairofpagestyles
\automark[section]{subsection}

\renewcommand{\subsectionmark}[1]{}

\lohead{{\small \headertitle}}
\rohead{{\small \headerauthors}}

\RedeclareSectionCommand[%
  font=\Large\sffamily\bfseries,%
  beforeskip=1\baselineskip,%
  afterskip=0.5\baselineskip,%
  indent=0em%
  ]{section}

\RedeclareSectionCommands[%
  font=\normalfont\bfseries,%
  afterskip=-1em%
  ]{subsection,subsubsection}

\RedeclareSectionCommands[%
  font=\normalfont\itshape,%
  afterskip=-1em,%
  indent=0pt,
  ]{paragraph}

\ifbool{nobiblatex}{}{%
  \AtEveryBibitem{\clearfield{doi}}
  \AtEveryBibitem{\clearfield{isbn}}
  \AtEveryBibitem{\clearfield{issn}}
  \AtEveryBibitem{\clearfield{pages}}
  \AtEveryBibitem{\clearlist{language}}

  \setlength\bibitemsep{1pt}
  
  \renewbibmacro*{in:}{}
  \DeclareFieldFormat
    [article,inbook,incollection,inproceedings,patent,thesis,unpublished]
    {title}{#1}
}

\newenvironment{enumeratearabic*}{
\begin{enumerate*}[label=(\arabic*)] 
}{
\end{enumerate*}
}

\newenvironment{enumerateroman*}{
\begin{enumerate*}[label=(\roman*)] 
}{
\end{enumerate*}
}

\numberwithin{equation}{section}

\theoremnumbering{arabic}
\newtheorem{theoremcounter}{theoremcounter}[section]
\theoremnumbering{Alph}
\newtheorem{maintheoremcounter}{maintheoremcounter}

\theoremstyle{plain}

\newtheorem{corollary}[theoremcounter]{Corollary}
\newtheorem{lemma}[theoremcounter]{Lemma}

\newtheorem{proposition}[theoremcounter]{Proposition}
\newtheorem{theorem}[theoremcounter]{Theorem}

\theoremstyle{plain}

\newtheorem{maintheorem}[maintheoremcounter]{Theorem}

\theoremstyle{definition}

\newtheorem{definition}[theoremcounter]{Definition}

\theoremstyle{remark}

\newtheorem{remark}[theoremcounter]{Remark}

\theoremstyle{nonumberremark}

\newcommand{\tx}{\ensuremath{\text}}
\newcommand{\thdash}{\nbd th}
\newcommand{\nbd}{\nobreakdash-\hspace{0pt}}

\ifdraft{
 \newcommand{\texpdf}[2]{#1}
}{
 \newcommand{\texpdf}[2]{\texorpdfstring{#1}{#2}}
}

\newcommand{\tbf}{\bfseries}

\newcommand{\bbone}{\ensuremath{\mathds{1}}}

\newcommand{\cE}{\ensuremath{\mathcal{E}}}

\newcommand{\cH}{\ensuremath{\mathcal{H}}}

\newcommand{\cO}{\ensuremath{\mathcal{O}}}

\newcommand{\rmd}{\ensuremath{\mathrm{d}}}

\newcommand{\rmB}{\ensuremath{\mathrm{B}}}

\newcommand{\rmE}{\ensuremath{\mathrm{E}}}

\newcommand{\rmH}{\ensuremath{\mathrm{H}}}

\newcommand{\rmL}{\ensuremath{\mathrm{L}}}
\newcommand{\rmM}{\ensuremath{\mathrm{M}}}

\newcommand{\rmS}{\ensuremath{\mathrm{S}}}

\newcommand{\rmZ}{\ensuremath{\mathrm{Z}}}

\newcommand{\ra}{\ensuremath{\rightarrow}}

\newcommand{\lra}{\ensuremath{\longrightarrow}}

\newcommand{\mto}{\ensuremath{\mapsto}}
\newcommand{\lmto}{\ensuremath{\longmapsto}}

\renewcommand{\Re}{\ensuremath{\mathrm{Re}}}
\renewcommand{\Im}{\ensuremath{\mathrm{Im}}}

\newcommand{\isdiv}{\ensuremath{\mathop{\mid}}}

\renewcommand{\pmod}[1]{\ensuremath{\;(\mathrm{mod}\, #1)}}

\newcommand{\Hom}{\ensuremath{\mathrm{Hom}}}

\newenvironment{psmatrix}{\left(\begin{smallmatrix}}{\end{smallmatrix}\right)}

\newcommand{\lspan}{\ensuremath{\mathop{\mathrm{span}}}}

\newcommand{\ZZ}{\ensuremath{\mathbb{Z}}}
\newcommand{\QQ}{\ensuremath{\mathbb{Q}}}
\newcommand{\RR}{\ensuremath{\mathbb{R}}}
\newcommand{\CC}{\ensuremath{\mathbb{C}}}

\newcommand{\SL}[1]{\ensuremath{\mathrm{SL}_{#1}}}

\newcommand{\sym}{\ensuremath{\mathrm{sym}}}

\newcommand{\HS}{\mathbb{H}}

\let\phi\varphi

\newcommand{\sfd}{\mathsf{d}}

\newcommand{\ga}{\gamma}
\newcommand{\Ga}{\Gamma}

\newcommand{\symd}{\sym^\sfd}

\newcommand{\para}{\mathrm{pb}}

\newcommand{\Bpara}{\rmB^1_{\para}}
\newcommand{\Zpara}{\rmZ^1_{\para}}
\newcommand{\Hpara}{\rmH^1_{\para}}

\newcommand{\pilow}{\pi^{\mathrm{low}}}

\renewcommand{\Im}{\operatorname{Im}} 

\mathchardef\xboxplus=\numexpr \boxplus-"2000\relax
\renewcommand{\boxplus}{\mathbin{\mathop{\xboxplus}}}

\newcommand{\cdotop}{\;\cdot\;}

\newcommand{\headertitle}{{\normalfont%
  Eichler integrals and generalized second order Eisenstein series
}}
\newcommand{\headerauthors}{%
  A.~Ahlb\"ack,
  T.~Magnusson,
  M.~Raum%
}
\title{%
  Eichler integrals and\\generalized second order Eisenstein series
}
\author{%
  Albin Ahlb\"ack%
\and
  Tobias Magnusson%
\and
  Martin Raum%
\thanks{The author was partially supported by Vetenskapsr\aa det Grant~2015-04139 and~2019-03551.}%
}

\begin{document}

\thispagestyle{scrplain}
\begingroup
\deffootnote[1em]{1.5em}{1em}{\thefootnotemark}
\maketitle
\endgroup


\begin{abstract}
\small
\noindent
{\tbf Abstract:}
We show that all Eichler integrals, and more generally all ``generalized second order modular forms'' can be expressed as linear combinations of corresponding generalized second order Eisenstein series with coefficients in classical modular forms. We determine the Fourier series expansions of generalized second order Eisenstein series in level one, and provide tail estimates via convexity bounds for additively twisted $\rmL$-functions. As an application, we illustrate a bootstrapping procedure that yields numerical evaluations of, for instance, Eichler integrals from merely the associated cocycle. The proof of our main results rests on a filtration argument that is largely rooted in previous work on vector-valued modular forms, which we here formulate in classical terms.
\\[.3\baselineskip]
\noindent
\textsf{\textbf{%
  second order modular forms%
}}%
\noindent
\ {\tiny$\blacksquare$}\ %
\textsf{\textbf{%
  period polynomials%
}}%
\noindent
\ {\tiny$\blacksquare$}\ %
\textsf{\textbf{%
  Rankin-Selberg convolution%
}}
\\[.2\baselineskip]
\noindent
\textsf{\textbf{%
  MSC Primary:
  11F11
}}%
\ {\tiny$\blacksquare$}\ %
\textsf{\textbf{%
  MSC Secondary:
  11F30, 11F75
}}
\end{abstract}




\Needspace*{4em}
\addcontentsline{toc}{section}{Introduction}
\markright{Introduction}

\lettrine[lines=2,nindent=.2em]{\tbf I}{n} recent years, there has been a surge of interest in the study of iterated Eichler-Shimura integrals, see for example~\cite{brown-2018-2,brown-2018,brown-2020,diamantis-2020-preprint}. This has in part been motivated by the fact that they are closely related to the string theoretic notion of holomorphic graph functions and modular graph functions. Indeed, the $A$-cycle graph functions and $B$-cycle graph functions introduced by Broedel-Schlotterer-Zerbini~\cite{broedel-schlotterer-zerbini-2019} can be expressed in terms of iterated Eisenstein integrals. Diamantis~\cite{diamantis-2020-preprint} showed that iterated Eichler-Shimura integrals are also closely related to higher order modular forms, first introduced by Goldfeld~\cite{goldfeld-1999}. In a similar vein Mertens--Raum~\cite{mertens-raum-2021} showed that both higher order modular forms and iterated Eichler-Shimura integrals can be reinterpreted as components of vector-valued modular forms of higher depth arithmetic types. In this paper, we explore the problem of efficiently computing Eichler-Shimura integrals by means of the aforementioned theoretical frameworks, focusing on the special case of Eichler integrals of cusp forms of level one. Our hope is that this approach can eventually be generalized to provide alternative means of computing point evaluations and Fourier series expansions of modular graph functions, for example considered by D'Hoker-Duke in~\cite{dhoker-duke-2018}.

We offer two theorems, that demonstrate the viability of our approach. They also prepare the route to a more general framework (see Remark~\ref{rm:thm:snd_order_eisenstein_saturation}). In Section~\ref{sec:g2mf}, we introduce the notion of \emph{generalized second order modular forms}, and show that Eichler integrals of cusp forms are examples. Generalized second order modular forms are holomorphic functions on~$\HS$ subject to an appropriate growth condition and whose modular deficits are cocycles taking values in polynomials whose coefficients are modular forms. Their name is motivated by the fact that they generalize second order modular forms. Indeed, the modular deficits of second order modular forms are cocycles with values in modular forms.

The notion of Eisenstein series extends from classical modular forms to generalized second order modular forms. In Section~\ref{sec:eis}, we prove our first main theorem, which provides Fourier series expansions for the generalized second order Eisenstein series $E_k^{[1]}(\tau;\phi)$ and $E_k^{[1]}(\tau;\phi,j)$, where $\phi$ is a parabolic cocycle for the symmetric power representation~$\symd(X)$ or it dual~$\symd(X)^\vee$ (see Section~\ref{ssec:preliminaries:vector_valued_symd}), respectively. 
\begin{maintheorem}
\label{mainthm:g2es}
Let~$\sfd\geq 0$ and~$k\geq 5 + \sfd$ be even integers, and let~$0\leq j\leq\sfd$ be an integer. Then if\/~$\phi$ and~$\phi^\vee$ are parabolic cocycles for $\symd(X)$ and\/~$\symd(X)^\vee$, respectively, we have
\begin{gather*}
  E_k^{[1]}(\tau;\, \phi)
\;=\;
  \sum_{r = 0}^{\sfd}
  (X - \tau)^r\,
  \sum_{n = 1}^\infty c(n)_r e(n \tau)
\quad\tx{and}\quad
  E_k^{[1]}(\tau; \phi^\vee,j)
=
  \sum_{n = 1}^\infty c(n) e(n \tau)
\tx{,}
\end{gather*}
where
\begin{align*}
  c(n)_r
\;&{}=\;
  \sum_{\substack{\gamma \in \Ga_\infty \backslash \SL{2}(\ZZ) \slash \Ga_\infty\\\ga \not\in \Ga_\infty}}
  \frac{e(n \frac{d}{c})}{c^k}\,
  \sum_{j = r}^{\sfd}
  \phi(\gamma^{-1})_{j}\;
  \sum_{l = r}^j
  \frac{(-2 \pi i)^{k - j + l}}{(k - j + l - 1)!}
  \binom{l}{r}
  \binom{j}{l}\,
  \big(-\tfrac{d}{c}\big)^{l-r}
  n^{k - j + l - 1}
\quad\tx{and}
\\
  c(n)
\;&{}=\;
  \sum_{\substack{\gamma \in \Ga_\infty \backslash \SL{2}(\ZZ) \slash \Ga_\infty\\\ga \not\in \Ga_\infty}}
  \frac{(-1)^{j-k} e(n \tfrac{d}{c})}{c^{k+2j-\sfd}}
  \sum_{r = 0}^j
  \binom{j}{r}
  \frac{(2 \pi i)^{k + r}}{(k + r - 1)!}\,
  \phi^\vee(\gamma^{-1}) \Big( \big( X + \tfrac{d}{c} \big)^{\sfd-j+r} \Big)
  n^{k + r - 1}
\tx{,}
\end{align*}
and~$(c,d)$ are the bottom entries of the double coset representative $\ga$.
\end{maintheorem}
Furthermore, we provide explicit bounds for the tails of the Fourier series expansions of $E_k^{[1]}(\tau;\phi)$ and $E_k^{[1]}(\tau;\phi,j)$ in Section~\ref{ssec:eval_g2es}, which we employ in our primary example in Section~\ref{sec:bootstrap}.

With their Fourier series expansions at hand, it becomes interesting to consider which generalized second order modular forms can be expressed in terms of these Eisenstein series. In Section~\ref{sec:saturation}, we prove our second main theorem which shows that any generalized second order modular form of type $(\bbone,\symd)$ or $(\symd,\bbone)$ can, after multiplication with some power of the discriminant modular form $\Delta$, be expressed as a linear combination of products of Eisenstein series and generalized second order Eisenstein series, together with a remainder term that is a classical modular form or a vector-valued modular form for $\symd$. The modular remainder term can be expressed in terms of Eisenstein series~\cite{raum-xia-2020} or by any other means.

The formal statement of the theorem features the saturation of ideals~\cite{cox-little-oshea-2015} (see also Section~\ref{sec:saturation}), the spaces of generalized second order modular forms from Definition~\ref{def:generalized_second_order_modular_forms}, and the spaces of generalized second order Eisenstein series that we provide in~\eqref{eq:def:generalized_second_order_eisenstein_modules}.

\begin{maintheorem}
\label{mainthm:saturation}
Given integers~$0 \le \sfd$ and~$k_0 \geq 5 + \sfd$, we have
\begin{alignat*}{3}
  \big(
  \rmE^{\mathrm{pure}[1]}_{\ge k_0}(\symd(X), \bbone)
  +
  \rmM_{\bullet,\sfd}
  \,:\, \Delta^\infty
  \big)
\;&{}=\;
  \rmM^{[1]}_\bullet(\symd(X), \bbone)
\quad\tx{and}
\\
  \big(
  \rmE^{\mathrm{pure}[1]}_{\ge k_0}(\bbone, \symd(X))
  +
  \rmM_\bullet
  \,:\, \Delta^\infty
  \big)
\;&{}=\;
  \rmM^{[1]}_\bullet(\bbone, \symd(X))
\tx{.}
\end{alignat*}
\end{maintheorem}
Theorem~\ref{mainthm:saturation} is proved in Section~\ref{sec:saturation} using Rankin's product identity and a filtration for vector-valued modular forms of type $\symd$.

In Section~\ref{sec:bootstrap}, we illustrate how to combine Theorem~\ref{mainthm:g2es} and Theorem~\ref{mainthm:saturation} to obtain expressions for Eichler integrals of cusp forms in terms of generalized second order Eisenstein series, using only their modular deficit and potentially a small number of point evaluations. These expressions allow for an efficient evaluation at any point on the Poincar\'e upper half plane. We also provide a concrete example of this method applied to the Eichler integral~$\cE(\Delta)$ of~$\Delta$. It should be noted that, in contrast to modular graph functions, one can compute the Fourier series expansion of $\cE(\Delta)$ directly and then evaluate this series. However, our method works for all generalized second order modular forms.

To be precise, we provide expressions for $\Delta^2\cE(\Delta)$ and $\Delta^5\cE(\Delta)$ in terms of sums of products of Eisenstein series and pure generalized second order Eisenstein series, together with a modular remainder term. In the former case, the low weight forces the remainder term to vanish. This makes it possible to directly evaluate $\Delta^2 \cE(\Delta)$ given only its modular deficit. Despite the appeal of this, the low weight also implies that the evaluations of the Eisenstein series are slow with comparatively large error terms. In the case of $\Delta^5 \cE(\Delta)$, the Eisenstein series can be evaluated more quickly, and the improvement in convergence is exponential in~$n$ when considering~$\Delta^n \cE(\Delta)$, where $n \ge 2$. The modular remainder term, however, will not vanish in this situation. We can determine it by evaluating~$\Delta^5 \cE(\Delta)$ at finitely many points. In the general setting of~$\Delta^n\cE(\Delta)$, the number of evaluations required grows linearly in~$n$. These evaluations can be performed using a previously computed expression for~$\Delta^2 \cE(\Delta)$. That is, we use a few slow evaluations to obtain an expression that can be evaluated faster. We think of this procedure as analogous to the statistical concept of ``bootstrapping''~\cite{efron-1993}.

In a forthcoming paper, we plan to develop our method so that it can be applied directly to various modular graph functions.


\section{Preliminaries}

In this section we summarize the notation that we use throughout the paper, and describe the fundamental notions of vector-valued modular forms. For the basics of classical modular forms, we refer to the book by Miyake~\cite{miyake-1989}, and the book by Diamond and Shurman~\cite{diamond-shurman-2005}. We adopt some notation from~\cite{diamond-shurman-2005} without further mentioning; for example, $\Ga_\infty$ denotes the parabolic subgroup of~$\SL{2}(\ZZ)$ including negative identity, and $\rmM_k$ and~$\rmS_k$ are spaces of modular forms and of cusp forms of weight $k$ and level one.

\subsection{Vector-valued modular forms}

Let $S=\begin{psmatrix}0&-1\\1&0\end{psmatrix}$ and $T=\begin{psmatrix}1&1\\0&1\end{psmatrix}$ be the standard generators for $\SL2(\ZZ)$. An arithmetic type is a finite-dimensional complex representation $\rho$ of $\SL2(\ZZ)$. We denote the representation space of $\rho$ by $V(\rho)$ and write $\rho^\vee$ for the dual of $\rho$.

Let $k$ be an integer and let $\rho$ be an arithmetic type. Let $f:\HS\to V(\rho)$ be a function and let $\|\cdot\|$ be a norm on $V(\rho)$. If there exists a number $a\in\RR$ such that for all $\ga\in\SL2(\ZZ)$ we have uniformly in $\Re(\tau)$ that
\begin{gather*}
  \big\| \big( f\big|_k\ga \big) (\tau) \big\|
=
  \cO(\Im(\tau)^a)
\quad\tx{as\ }
  \Im(\tau)\to\infty
\tx{,}
\end{gather*}
then we say that $f$ has moderate growth.

Given a function~$f :\, \HS \ra V(\rho)$, $k \in \ZZ$, and an arithmetic type~$\rho$, we write
\begin{gather*}
  \big( f\big|_{k,\rho}\gamma \big) (\tau)
=
  (c\tau+d)^{-k} \rho(\gamma^{-1}) f(\gamma\tau)
\tx{,}\quad
\tx{where\ }
  \gamma = \begin{psmatrix}a&b\\c&d\end{psmatrix} \in \SL2(\ZZ).
\end{gather*}
A vector-valued modular form of weight $k$ and type $\rho$ is a holomorphic function~$f :\, \HS \ra V(\rho)$ of moderate growth that satisfies $f|_{k,\rho}\gamma=f$ for all~$\ga\in\SL2(\ZZ)$. The space of vector-valued modular forms of type $\rho$ and weight $k$ is denoted by $\rmM_k(\rho)$.

\subsection{Vector-valued modular forms of type \texpdf{$\symd$}{symd}}
\label{ssec:preliminaries:vector_valued_symd}

Let $\sfd$ be a non-negative integer and let $\CC[X]_\sfd$ be the space of polynomials in $X$ of degree at most $\sfd$. Then the $\sfd$\thdash{} symmetric power of the standard representation of $\SL2(\ZZ)$, denoted by $\symd(X)$, is defined by $V(\symd(X))=\CC[X]_\sfd$ and the action
\begin{gather*}
  \symd(X)(\gamma^{-1})p(X)
\;:=\;
  p(X) \big|_{-\sfd} \gamma
=
  (cX+d)^\sfd\, p\big(\mfrac{aX+b}{cX+d}\big)
\tx{,}\quad\tx{where\ }
  \ga=\begin{psmatrix}a&b\\c&d\end{psmatrix}
\tx{.}
\end{gather*}
The~$j$\thdash{} coefficient of~$p \in \CC[X]_\sfd$ will be written as~$p_j$. We have a self-duality given by the following pairing on~$\symd(X) \otimes \symd(X)$:
\begin{gather}
\label{eq:def:symd_self_duality_pairing}
  \langle p, q \rangle
\;:=\;
  \sum_{i=0}^\sfd
  (-1)^{\sfd-i} \mbinom{\sfd}{i}^{-1} p_i q_{\sfd-i}
\tx{.}
\end{gather}
Given a polynomial~$p(\tau) \in \CC_\sfd[X]$, we have $\langle p(X), (X-\tau)^\sfd \rangle = p(\tau)$.

We use the shorthand notation $|_{k,\sfd}$ for $|_{k,\symd(X)}$ and $\rmM_{k,\sfd}$ for $\rmM_k(\symd(X))$. The graded $\rmM_\bullet$\nbd module of vector-valued modular forms of type $\symd$ is given by
\begin{gather*}
  \rmM_{\bullet,\sfd}
\;:=\;
  \bigoplus_{k \in \ZZ} \rmM_{k,\sfd}
\tx{.}
\end{gather*}

The next result is due to Kuga--Shimura~\cite{kuga-shimura-1960}, and we merely rephrase it in our language. Given an integer~$0 \le j \le \sfd$, we consider the increasing filtration of~$\rmM_{k,\sfd}$ by the spaces
\begin{gather}
\label{eq:def:component_vanishing_filtration}
  \rmM_{k,\sfd}[j]
\;:=\;
  \Big\{
    \sum_{r=j}^\sfd (X - \tau)^r f_r \in \rmM_{k,\sfd}
  \;:\;
    f_j :\, \HS \ra \CC
  \Big\}
\end{gather}
of modular forms of type~$\symd$, whose lower~$(X-\tau)^r$\nbd terms vanish. Observe that
\begin{gather*}
  \rmM_{\bullet,\sfd}[j]
\;:=\;
  \bigoplus_{k \in \ZZ}
  \rmM_{k,\sfd}[j]
\end{gather*}
is a module for~$\rmM_{\bullet}$. Therefore the filtration in~\eqref{eq:def:component_vanishing_filtration} yields a filtration of~$\rmM_\bullet$\nbd modules.

\begin{lemma}[See Kuga--Shimura~\cite{kuga-shimura-1960}]
\label{la:projection_to_lowest_symd_component}
Let~$k$ and~$\sfd \ge 0$ be even integers. Then for every integer~$0 \le j \le \sfd$ there is a map
\begin{gather}
\label{eq:la:projection_to_lowest_symd_component}
  \pilow_{k,\sfd,j}
:\,
  \rmM_{k,\sfd}[j]
\lra
  \rmM_{k + 2j - \sfd}
\tx{,}\,
  \sum_{r=j}^{\sfd} (X - \tau)^r f_r
\lmto
  f_{j}
\tx{.}
\end{gather}
The map
\begin{gather*}
  \pilow_{\bullet,\sfd,j}
:\,
  \rmM_{\bullet,\sfd}[j]
\lra
  \rmM_\bullet
\end{gather*}
obtained by applying~\eqref{eq:la:projection_to_lowest_symd_component} in weight~$k$ is a homomorphism of\/~$\rmM_\bullet$\nbd modules.
\end{lemma}
\begin{proof}
The second part of the statement follows directly, since the~$\rmM_\bullet$\nbd module structure on~$\rmM_{\bullet,\sfd}$ is given by component-wise multiplication with respect to the basis~$(X - \tau)^r$, $0 \le r \le \sfd$ at~$\tau \in \HS$.

To prove that the given map is well-defined, we have to show that~$f_{j}$ is invariant under the slash action of~$\SL{2}(\ZZ)$ of weight~$k + 2j - \sfd$. Invariance of~$f_{j}$ under the action of~$T$, follows when employing invariance of the left hand side and comparing terms on the left and right hand side of
\begin{gather*}
  \sum_{r = 0}^\sfd (X - \tau)^r f_r
\;=\;
  \Big( \sum_{r = 0}^\sfd (X - \tau)^r f_r \Big) \Big|_{k,\sfd} T
\;=\;
  \sum_{r = 0}^\sfd (X - \tau)^r \big( f_r \big| T \big)
\tx{,}
\end{gather*}
where we have suppressed the weight from the slash action on the right hand side, since the action of~$T$ does not depend on it.

To verify invariance under~$S$, we compute
\begin{align*}
  \sum_{r = 0}^\sfd (X - \tau)^r f_r
\;&{}=\;
  \Big( \sum_{r = 0}^\sfd (X - \tau)^r f_r \Big) \Big|_{k,\sfd} S
\;=\;
  \sum_{r = 0}^\sfd (-1)^r
  ((X - \tau) + \tau)^{\sfd -r} (X - \tau)^r
  \big( f_r \big|_{k+r} S \big)
\\
\;{}&=\;
  \sum_{r = 0}^\sfd (-1)^r
  \sum_{i = 0}^{\sfd-r}
  \mbinom{\sfd-r}{i} 
  (X - \tau)^{r+i}
  \big( f_r \big|_{k+2r+i-\sfd} S \big)
\\
\;{}&=\;
  \sum_{i = 0}^\sfd
  (X - \tau)^i
  \sum_{r = 0}^i (-1)^r
  \mbinom{\sfd-r}{i-r} 
  \big( f_r \big|_{k+r+i-\sfd} S \big)
\tx{.}
\end{align*}
If~$f_r = 0$ for~$r < j$, then comparing the left and right hand side, the term with~$r = i = j$ reveals that~$f_{j} = f_{j} |_{k + j + j - \sfd}\, S$, finishing our proof.
\end{proof}

\section{Generalized second order modular forms}
\label{sec:g2mf}

In this section, we define generalized second order modular forms and show that holomorphic Eichler integrals are generalized second order modular forms. Our generalized second order modular forms are similar to the ``extended second order modular forms'' recently studied by Diamantis~\cite{diamantis-2020-preprint}. 

The next definition features the parabolic cohomology group~$\Hpara(\Ga, \sigma^\vee \otimes \rho)$, which we revisit in Section~\ref{ssec:second_order:cohomology} for completeness.

\begin{definition}
\label{def:generalized_second_order_modular_forms}
Let $\rho$ and $\sigma$ be arithmetic types for a finite index subgroup $\Gamma\subseteq\SL2(\ZZ)$, and let $k$ be an integer. We say that a holomorphic function $f:\, \HS \ra V(\rho)$ of moderate growth is a generalized second order modular form of weight~$k$ and type~$(\rho,\sigma)$ if there exists a finite subset $A \subseteq \Zpara(\Gamma, \sigma^\vee\otimes\rho)$ and a corresponding set of holomorphic modular forms $g_\phi \in \rmM_k(\sigma)$, $\phi \in A$, such that for every $\gamma\in\Gamma$, we have
\begin{gather}
\label{eq:def:generalized_second_order_modular_forms}
  \big( f \big|_{k,\rho} (\gamma-1) \big)(\tau)
\;=\;
  \sum_{\phi\in A}
  \phi\big( \gamma^{-1} \big) \big( g_\phi(\tau) \big)
\tx{.}
\end{gather}
\end{definition}

We write~$\rmM_k^{[1]}(\rho,\sigma)$ for the space of generalized second order modular forms of weight~$k$ and arithmetic type~$(\rho,\sigma)$. We have an inclusion
\begin{gather}
  \rmM_k(\rho)
\;\subseteq\;
  \rmM_k^{[1]}(\rho,\sigma)
\tx{.}
\end{gather}

Generalized second order modular forms of type $(\bbone,\bbone)$ are the same as second order modular forms, introduced by Goldfeld~\cite{goldfeld-1999}. It is possible to view generalized second order modular forms as components of vector-valued modular forms of suitable arithmetic type~\cite{mertens-raum-2021}.

\subsection{Cohomology}
\label{ssec:second_order:cohomology}

Consider arithmetic types~$\rho$ and~$\sigma$. Throughout this work, we identify~$V(\sigma)^\vee \otimes V(\rho)$ with~$\Hom(V(\sigma), V(\rho))$. We denote the group of $1$-coboundaries and $1$-cocycles by
\begin{align*}
  \rmB^1(\SL{2}(\ZZ),\, \rho)
&\;=\;
  \big\{ f : \SL{2}(\ZZ) \to V(\rho),
  \ga \mto \rho(\ga) v - v \,:\,
  v \in V(\rho)
  \big\}
\tx{,}
\\
  \rmZ^1(\SL{2}(\ZZ),\, \rho)
&\;=\;
  \big\{ f : \Gamma \to V(\rho) \,:\,
  \forall \gamma_1,\gamma_2 \in \Gamma \,.\,
  f(\gamma_1\gamma_2) = \rho(\gamma_1) f(\gamma_2) + f(\gamma_1)
  \big\}
\tx{.}
\end{align*}
Cocycles and coboundaries that vanish on all parabolic elements of~$\Gamma$ are called parabolic. We denote the corresponding subgroups by~$\Bpara(\SL{2}(\ZZ), \rho)$ and~$\Zpara(\SL{2}(\ZZ), \rho)$. The quotient
\begin{gather*}
  \Hpara(\SL{2}(\ZZ),\, \rho)
\;:=\;
  \Zpara(\SL{2}(\ZZ),\, \rho) \big\slash \Bpara(\SL{2}(\ZZ),\, \rho)
\end{gather*}
is called the $1$-st parabolic cohomology group with coefficients in~$\rho$. We identify cohomology classes with representatives via a section to the quotient map from cocycles to cohomology classes, that we fix once and for all.

Recall the pairing on~$V(\symd) \otimes V(\symd)$ in~\eqref{eq:def:symd_self_duality_pairing}. Given~$\phi \in \Zpara(\SL{2}(\ZZ), \symd(X))$, we define a ``dual'' cocycle~$\phi^\vee \in \Zpara(\SL{2}(\ZZ), \symd(X)^\vee)$ by
\begin{gather}
\label{eq:def:dual_cocycle}
  \phi^\vee(\ga)(v)
\;:=\;
  \langle \phi(\ga), v \rangle
\tx{,}\qquad
  \phi^\vee(\ga^{-1})
=
  \sum_{i = 0}^\sfd
  (-1)^{\sfd-i} \mbinom{\sfd}{i}^{-1}
  \phi(\ga^{-1})_{\sfd-i}
  (X^i)^\vee
\tx{,}
\end{gather}
where~$X^{\vee\,i}$, $0 \le i \le \sfd$, is the dual basis of~$X^i$, $0 \le i \le \sfd$.

Given a holomorphic function~$f :\, \HS \ra \CC$ and~$k \in \ZZ$, we have a $1$-coboundary
\begin{gather}
\label{eq:def:cocycle_modular_deficit}
  \phi_f(\gamma) \;:=\; f \big|_k (\gamma^{-1} - 1)
\,\in\,
  \rmB^1\big( \SL{2}(\ZZ),\, \cH_k(\HS) \big)
\tx{,}
\end{gather}
where~$\cH_k(\HS)$ is space of holomorphic functions from~$\HS$ to~$\CC$ equipped with the weight-$k$ slash action of~$\SL{2}(\ZZ)$. Note that~$\phi_f$ vanishes if and only if~$f$ is modular invariant of weight~$k$ for $\SL2(\ZZ)$.

\subsection{Eichler integrals}
\label{ssec:eich_int_g2mf}

Given a positive, even integer~$k$, recall the Eichler integral associated with a cusp form~$f \in \rmS_k$:
\begin{gather}
\label{eq:def:eichler_integral}
  \mathcal{E}(f)(\tau)
\;:=\;
  \int_\tau^{i\infty}
  f(z) (\tau-z)^{k-2} \rmd z
\tx{.}
\end{gather}
This integral is absolutely convergent and holomorphic in~$\tau$. The cocycle attached to~$\cE(f)$  is given by
\begin{gather}
\label{eq:cocycle_eichler_integral}
  \phi_{\mathcal{E}(f)}(\ga)
=
  \mathcal{E}(f) \big|_{2-k}
  (\gamma^{-1}-1)
\in
  \Zpara\big( \SL{2}(\ZZ), \sym^{k-2}(\tau) \big)
\tx{.}
\end{gather}
Notice that while it is a co-boundary when viewed as a cocycle with coefficients in~$\cH_{2-k}(\HS)$ as in~\eqref{eq:def:cocycle_modular_deficit}, it is nontrivial in general when viewed as a cocycle with coefficients in~$\sym^{k-2}(\tau)$.

The value of~\eqref{eq:cocycle_eichler_integral} at~$S$ can be described in terms of special~$\rmL$-values for~$f$ (see for example~\cite{kohnen-zagier-1984}):
\begin{gather*}
  \phi_{\mathcal{E}(f)}(S)
=
  \sum_{j=0}^{k-2}
  \tau^{k-2-j}\,
  \frac{(k-2)!}{(k-2-j)!\, (2\pi i)^{j+1}}
  \rmL(f;j+1)\,
\tx{.}
\end{gather*}
A similar consideration shows that for~$\ga = \begin{psmatrix} a & b \\ c & d \end{psmatrix} \in \SL{2}(\ZZ)$, we have
\begin{gather}
\label{eq:cocycle_explicit_expression}
  \phi_{\mathcal{E}(f)}( \ga^{-1} )
=
  \sum_{j = 0}^{k-2}
  \tau^{k-2-j}\,
  \frac{(k-2)!}{(k-2-j)!}
  \sum_{r = 0}^j
  \frac{( d \slash c )^{j-r}}{(j-r)!\, (2 \pi i)^{r+1}}
  \rmL\big(f, \mfrac{-d}{c}; r+1 \big)
\tx{,}
\end{gather}
where the additively twisted~$\rmL$\nbd function of~$f$ is defined via analytic continuation of
\begin{gather*}
  \rmL\big(f, \mfrac{-d}{c}; s \big)
:=
  \sum_{n = 1}^\infty
  e\big(\mfrac{-d}{c}n\big) \frac{c(f;n)}{n^s}
\tx{.}
\end{gather*}

Leveraging the cocycle in~\eqref{eq:cocycle_eichler_integral}, we find that Eichler integrals are generalized second order modular forms. The next proposition can be viewed as a special case of Theorem~3.7 of~\cite{mertens-raum-2021} when using the language of vector-valued modular forms.

\begin{proposition}
\label{prop:eichler_integral}
Given a positive, even integer~$k$, consider a cusp form~$f \in \rmS_k$. Then
\begin{gather*}
  \cE(f)
\in
  \rmM^{[1]}_{2-k}\big( \bbone, \sym^{k-2}(X) \big)
\tx{.}
\end{gather*}
\end{proposition}

\begin{proof}
We write~$\phi$ for the cocycle attached to~$\cE(f)$ in~\eqref{eq:cocycle_eichler_integral} and~$\phi^\vee$ for its dual defined in~\eqref{eq:def:dual_cocycle}. In Definition~\ref{def:generalized_second_order_modular_forms}, we set~$A = \{\phi^\vee\}$ and~$g_{\phi^\vee} = (X - \tau)^{k-2}$. Rephrasing~\eqref{eq:cocycle_eichler_integral}, we have
\begin{gather*}
  \big( \cE(f) \big|_{2-k} (\ga - 1) \big)(\tau)
=
  \phi(\ga^{-1})(\tau)
\tx{.}
\end{gather*}
To verify~\eqref{eq:def:generalized_second_order_modular_forms}, we have to see that
\begin{gather}
\label{eq:prop:eichler_integral:cocycle}
\phi(\ga^{-1})(\tau)
=
  \phi^\vee(\ga^{-1}) \big( (X - \tau)^{k-2} \big)
\tx{,}
\end{gather}
which follows directly from the definition of~$\phi^\vee$ in~\eqref{eq:def:dual_cocycle} and the relation stated after~\eqref{eq:def:symd_self_duality_pairing}.

Since~$(X - \tau)^{k-2}$ has moderate growth and~$\cE(f)(\tau) \to 0$ as $\tau \to i\infty$, we see that~$\cE(f)$ is of moderate growth. Since~$\cE(f)$ is holomorphic, we have finished the proof.
\end{proof}

\section{Eisenstein series}
\label{sec:eis}

In this section we introduce generalized second order Eisenstein series, which yield generalized second order modular forms in the sense of Definition~\ref{def:generalized_second_order_modular_forms}. In particular, they are associated with a pair of arithmetic types~$(\rho,\sigma)$. The construction is based on a parabolic $1$-cocycle for~$\sigma^\vee \otimes \rho$. We determine the Fourier series expansions in the special cases~$(\rho,\sigma)=(\bbone,\symd(X))$ and~$(\rho,\sigma)=(\symd(X),\bbone)$.

We start with vector-valued Eisenstein series of type~$\symd$ for a nonnegative integer~$\sfd$. If~$\sfd = 0$, it equals the classical Eisenstein series~$E_k$ of weight~$k$. Given integers~$k > 2 + \sfd$ and~$0 \le j \le \sfd$ with $k \equiv \sfd \,\pmod{2}$ we define
\begin{gather}
\label{eq:def:eisenstein_series_symd}
  E_k(\tau;\sfd,j)
\;:=\;
  \sum_{\gamma \in \Gamma_\infty \backslash \SL2(\ZZ)}
  (X - \tau)^j \big|_{k,\symd(X)} \gamma
\tx{.}
\end{gather}

The definition of generalized second order Eisenstein series requires the group of real unipotent upper triangular matrices:
\begin{gather*}
  \SL{2}(\RR)_\infty
\;:=\;
  \big\{
    \pm
    \begin{psmatrix}
      1 & b
    \\
      0 & 1
    \end{psmatrix}
    \in
    \SL{2}(\RR)
  \; : \;
    b \in \RR
  \big\}
\tx{.}
\end{gather*}
Note that~$(X-\tau)^j$ in~\eqref{eq:def:eisenstein_series_symd} is invariant under the slash action of~$\SL{2}(\RR)_\infty$.

\begin{definition}
\label{def:second_order_eisenstein_series}
Let $k$ be an integer, $\rho$ and $\sigma$ be arithmetic types, and~$\phi$ be a parabolic $1$-cocycle, i.e., $\phi \in \Zpara(\SL2(\ZZ), \sigma^\vee \otimes \rho)$. Assume that the restriction of~$\sigma$ to~$\Ga_\infty$ admits an extension, say~$\sigma_\infty$, to~$\SL{2}(\RR)_\infty$. Consider a smooth function~$f:\HS \to V(\sigma)$ that is covariant with respect to real translations:
\begin{gather*}
  \forall g \in \SL{2}(\RR)_\infty \,:\,
  f \big| g
=
  \sigma_\infty(g)
  f
\tx{.}
\end{gather*}
Assuming absolute and locally uniform convergence of the right hand side, the generalized second order Eisenstein series of weight $k$ and type $(\rho,\sigma)$ associated to $(\phi,f)$ is defined as
\begin{gather}
\label{eq:def:second_order_eisenstein_series}
  E_k^{[1]}(\tau;\, \phi,f)
:=
  \sum_{\gamma \in \Gamma_\infty \backslash \SL2(\ZZ)}
  \phi( \gamma^{-1} ) \Big(
  \big( f \big|_{k,\sigma} \gamma \big)(\tau)
  \Big)
\tx{.}
\end{gather}
\end{definition}

\begin{remark}
The Eisenstein series in~\eqref{eq:def:second_order_eisenstein_series} is a component of a suitable vector-valued Eisenstein series~\cite{mertens-raum-2021}.
\end{remark}

A usual argument via reordering summation yields the next statement.
\begin{proposition}
\label{prop:g2es_is_g2mf}
Let~$k$, $\rho$, $\sigma$, $\phi$, and~$f$ be as in Definition~\ref{def:second_order_eisenstein_series}. Assume that~\eqref{eq:def:second_order_eisenstein_series} and the vector-valued Eisenstein series~$\sum f |_{k,\sigma} \ga$ of type~$\sigma$ converge absolutely and locally uniformly. If~$f$ is holomorphic and ~$E_k^{[1]}(\tau;\,\phi,f)$ has moderate growth, then we have
\begin{gather}
  E_k^{[1]}(\tau;\,\phi,f)
\in
  \rmM_k^{[1]}(\rho,\sigma)
\tx{,}
\end{gather}
\end{proposition}
\begin{proof}
As a locally uniform limit of holomorphic functions, $E_k^{[1]}(\cdotop;\, \phi, f)$ is holomorphic. Given~$\delta \in \SL{2}(\ZZ)$, continuity and the cocycle relation for~$\phi$ yield
\begin{align*}
&
  \sum_{\gamma \in \Gamma_\infty \backslash \SL2(\ZZ)}
  \phi( \gamma^{-1} ) \Big(
  f \big|_{k,\sigma} \gamma
  \Big) \big|_{k,\rho} \delta
\;=\;
  \sum_{\gamma \in \Gamma_\infty \backslash \SL2(\ZZ)}
  \rho(\delta^{-1})
  \phi( \gamma^{-1} ) \Big(
  \sigma(\ga^{-1}) f \big|_k \gamma \delta
  \Big)
\\
=\;&
  \sum_{\gamma \in \Gamma_\infty \backslash \SL2(\ZZ)}
  \phi( \delta^{-1} \gamma^{-1} ) \Big(
  \sigma(\delta^{-1}) \sigma(\ga^{-1}) f \big|_k \gamma \delta
  \Big)
  \,-\,
  \sum_{\gamma \in \Gamma_\infty \backslash \SL2(\ZZ)}
  \phi( \delta^{-1} ) \Big(
  \sigma(\delta^{-1})
  \sigma(\ga^{-1}) f \big|_k \gamma \delta
  \Big)
\\
=\;&
  \sum_{\gamma \in \Gamma_\infty \backslash \SL2(\ZZ)}
  \phi( \gamma^{-1} ) \Big(
  \sigma(\ga^{-1}) f \big|_k \gamma
  \Big)
  \,-\,
  \phi( \delta^{-1} ) \Bigg(
  \sum_{\gamma \in \Gamma_\infty \backslash \SL2(\ZZ)}
  \sigma(\delta^{-1})
  \sigma(\ga^{-1}) f \big|_k \gamma \delta
  \Bigg)
\tx{.}
\end{align*}
The contribution of~$\delta$ in the argument of the second term can be discarded by modular invariance, thus finishing the proof.
\end{proof}

Given even integers~$k$ and~$\sfd \geq 0$, and an integer~$0 \leq j \leq \sfd$, the constant functions and the functions~$(X-\tau)^j$ satisfy the condition on~$f$ in Definition~\ref{def:second_order_eisenstein_series} for the arithmetic types~$\sigma = \bbone$ and~$\sigma = \symd(X)$, respectively. By Proposition~5.4 of~\cite{mertens-raum-2021} the associated generalized second order Eisenstein series are defined for~$k \ge 5 + \sfd$ and yield functions of moderate growth, i.e., generalized second order modular forms. Note that one can improve this bound when employing convexity of additively twisted $\rmL$-functions as in Lemma~\ref{eq:la:cocycle_estimate} or in~\cite{diamantis-2020-preprint}. We use the following shorthand notation:
\begin{gather}
\label{eq:def:second_order_eisenstein_series_sym_cases}
\begin{alignedat}{2}
  E_k^{[1]}(\tau;\phi)
&\;:=\;
  E_k^{[1]}(\tau;\phi,1)
\tx{,}\quad
&&
  \text{if\ }(\rho,\sigma)=(\symd(X),\bbone)
\tx{;}
\\
  E_k^{[1]}(\tau;\phi,j)
&\;:=\;
  E_k^{[1]}(\tau;\phi,(X-\tau)^j)
\tx{,}\quad
&&
  \text{if\ }(\rho,\sigma)=(\bbone,\symd(X))
\tx{.}
\end{alignedat}
\end{gather}

\subsection{Fourier series expansion of~\texpdf{$E_k(\cdotop; \sfd, j)$}{Ekdj}}

Note that the function~$X - \tau$ is invariant under the slash action of any weight~$k$ and type~$\symd$, $\sfd \ge 0$. In particular, holomorphic modular forms of type~$\symd$ for~$\SL{2}(\ZZ)$, including these Eisenstein series, afford a Fourier series expansion of the form
\begin{gather}
\label{eq:symd_fourier_series}
  \sum_{r = 0}^{\sfd}
  (X - \tau)^r\,
  \sum_{n = 0}^\infty c(n)_r e(n \tau)
\tx{,}\quad
  c(n)_r \in \CC
\tx{.}
\end{gather}
The next proposition identifies the Fourier coefficients in case of the Eisenstein series of type~$\symd$. Its Corollary~\ref{cor:eisenstein_series_symd_filtration_steps} is instrumental in Section~\ref{sec:saturation}.

\begin{proposition}
\label{prop:eisenstein_series_symd_fourier_expansion}
Given even integers~$\sfd \ge 0$ and~$k > 2 + \sfd$, and an integer~$0 \le j \le \sfd$, we have
\begin{gather}
\label{eq:prop:eisenstein_series_symd_fourier_expansion}
  E_k(\tau;\, \sfd, j)
\;=\;
  (X - \tau)^j
  +
  \sum_{r = j}^{\sfd}
  (X - \tau)^r\,
  \sum_{n = 1}^\infty c(n)_r e(n \tau)
\end{gather}
with Fourier coefficients for~$j \le r \le \sfd$ and~$n \ge 1$:
\begin{gather*}
  c(n)_r
\;=\;
  \mbinom{\sfd-j}{r-j}
  \frac{(-2\pi i)^{k+j+r-\sfd}}{(k+j+r-\sfd-1)!\, \zeta(k+2j-\sfd)}\,
  n^{r - j} \sigma_{k+2j-\sfd-1}(n)
\tx{.}
\end{gather*}
\end{proposition}

\begin{corollary}
\label{cor:eisenstein_series_symd_filtration_steps}
For even integers~$\sfd \ge 0$ and~$k > 2 + \sfd$, and for an integer~$0 \le j \le \sfd$, we have
\begin{gather*}
  E_k(\,\cdot\,;\, \sfd, j)
\in
  \rmM_{k,\sfd}[j]
\tx{.}
\end{gather*}
The map~$\pilow_{k,\sfd,j}$ from~$\rmM_{k,\sfd}[j]$ to~$\rmM_{k + 2j - \sfd}$ sends it to the usual Eisenstein series~$E_{k + 2j - \sfd}$.
\end{corollary}

\begin{proof}%
[Proof of Proposition~\ref{prop:eisenstein_series_symd_fourier_expansion}]
Given $\gamma = \begin{psmatrix} a & b \\ c & d \end{psmatrix} \in \SL{2}(\ZZ)$ with $c \neq 0$, we find that $(X - \tau)^j |_{k,\sfd}\,\gamma$ equals
\begin{align*}
&
  (c X + d)^{\sfd-j}(c \tau + d)^{-k-j}\,
  \big( (a X + b) (c \tau + d) - (a \tau + b) (c X + d) \big)^j
\\
=\;{}&
  \frac{1}{c^{k+2j-\sfd}}\,
  \big((X - \tau) + (\tau + \tfrac{d}{c}) \big)^{\sfd-j}\,
  \big(\tau + \tfrac{d}{c}\big)^{-k-j}\,
  \big( (ac - ac) X \tau + (a d - b c) (X - \tau) + (bd - bd) \big)^j
\\
=\;{}&
  \frac{1}{c^{k+2j-\sfd}}\,
  \sum_{r = 0}^{\sfd-j}
  \mbinom{\sfd-j}{r}
  \frac{(X - \tau)^{j+r}}
       {(\tau + \frac{d}{c})^{k+j-(\sfd-j-r)}}
\;=\;
  \frac{1}{c^{k+2j-\sfd}}\,
  \sum_{r = j}^{\sfd}
  \mbinom{\sfd-j}{r-j}
  \frac{(X - \tau)^r}
       {(\tau + \frac{d}{c})^{k+j+r-\sfd}}
\tx{.}
\end{align*}

To apply the Lipschitz summation formula in the usual way, we extend~$E_k(\tau;\, \sfd, j)$ by the auxiliary factor~$\zeta(k+2j-d)\,\zeta(k+2j-d)^{-1}$ and split the defining summation over~$\Ga_\infty \backslash \SL{2}(\ZZ)$ into a sum over double cosets in~$\Ga_\infty \backslash \SL{2}(\ZZ) \slash \Ga_\infty$ with bottom entries~$c \in \ZZ$ and~$d \pmod{c}$ and over~$m \in \ZZ$. We thus find that~$E_k(\tau;\, \sfd, j)$ equals
\begin{align*}
&\hphantom{{}=\;}
  (X - \tau)^j
\,+\,
  \frac{1}{\zeta(k+2j-\sfd)}
  \sum_{r = j}^{\sfd}
  (X - \tau)^r\,
  \mbinom{\sfd-j}{r-j}
  \sum_{c = 1}^\infty
  \frac{1}{c^{k+2j-\sfd}}\,
  \sum_{\substack{d \pmod{c}\\m \in \ZZ}}
  \frac{1}
       {(\tau + \frac{d}{c} + m)^{k+j+r-\sfd}}
\\
&{}=
  (X - \tau)^j
\,+\,
  \frac{1}{\zeta(k+2j-\sfd)}
  \sum_{r = j}^{\sfd}
  (X - \tau)^r\,
  \mbinom{\sfd-j}{r-j}
  \sum_{c = 1}^\infty
  \frac{1}{c^{k+2j-\sfd}}\,
\\
&
  \hphantom{{}= (X - \tau)^j \,+\,}\qquad\qquad
  \sum_{d \pmod{c}}
  \frac{(-2\pi i)^{k+j+r-\sfd}}{(k+j+r-\sfd-1)!}
  \sum_{n = 1}^\infty
  n^{k+j+r-\sfd-1}
  e\big( n (\tau + \tfrac{d}{c})\big)
\\
&{}=
  (X - \tau)^j
\,+\,
  \sum_{r = j}^{\sfd}
  (X - \tau)^r\,
  \mbinom{\sfd-j}{r-j}
  \frac{(-2\pi i)^{k+j+r-\sfd}}{(k+j+r-\sfd-1)!\, \zeta(k+2j-\sfd)}
\\
&
  \hphantom{{}= (X - \tau)^j \,+\,
            \sum_{r = j}^{\sfd} (X - \tau)^r\,
	    \mbinom{\sfd-j}{r-j}}\qquad\qquad
  \sum_{c = 1}^\infty
  \sum_{n = 1}^\infty
  c^{r-j} n^{k+j+r-\sfd-1}
  e(c n \tau)
\tx{.}
\end{align*}
\end{proof}

\subsection{Evaluation of~\texpdf{$E_k(\cdotop; \sfd, j)$}{Ekdj}}

We use the Fourier series expansions of~$E_k(\cdotop; \sfd, j)$ to evaluate it at points~$\tau \in \HS$. This requires the following tail estimate.

\begin{lemma}
\label{la:errorbound_ekd_partial}
Given~$\tau = x + iy \in \HS$, integers~$a \ge 0$ and~$b > 0$, and an integer~$N \ge 1 + (a+b) \slash 2 \pi y$ we have
\begin{gather}
\label{eq:prop:errorbound_ekd_partial}
  \bigg|
    \sum_{n = N}^\infty n^a \sigma_b(n) e(n \tau)
  \bigg|
\;\le\;
  \frac{b\, \Ga(a+b+1,\, 2 \pi (N-1) y)}{(b-1)\, (2 \pi y)^{a+b+1}}
\tx{.}
\end{gather}
\end{lemma}

\begin{remark}
\label{rm:la:errorbound_ekd_partial_lambert_series}
Instead of employing the Fourier series expansion of~$E_k(\cdotop; \sfd, j)$ to evaluate it, one might rewrite it as a Lambert series using the relation
\begin{gather*}
  \sum_{n = 1}^\infty n^a \sigma_b(n) e(n \tau)
\;=\;
  \sum_{n = 1}^\infty A_a(e(n \tau)) \frac{n^{a+b} e(n \tau)}{(1 - e(n \tau))^{a + 1}}
\tx{,}
\end{gather*}
where~$A_a$ is the~$a$\thdash{} Eulerian polynomial. We do not employ this in our setting.
\end{remark}

\begin{proof}%
[Proof of Lemma~\ref{la:errorbound_ekd_partial}]
We have~$n^a \sigma_b(n) = n^{a+b} \sigma_{-b}(n)$ and the trivial bound using~$b > 0$
\begin{gather*}
  \sigma_{-b}(n)
\;=\;
  \sum_{d \isdiv n} d^{-b}
\;\le\;
  1 + \int_{1}^n t^{-b} \rmd\!t
\;=\;
  1 +
  \frac{1}{b - 1}
  -
  \frac{n^{1-b}}{(b - 1)}
\;<\;
  \frac{b}{b - 1}
\tx{.}
\end{gather*}
We insert this into the left hand side of~\eqref{eq:prop:errorbound_ekd_partial} to find
\begin{gather*}
  \bigg|
  \sum_{n = N}^\infty n^a \sigma_b(n) e(n \tau)
  \bigg|
\;\le\;
  \frac{b}{b-1}
  \sum_{n = N}^\infty
  n^{a+b} \exp(-2 \pi n y)
\tx{.}
\end{gather*}
The assumption~$N \ge 1 + (a+b) \slash 2 \pi y$ guarantees that the summand is monotone as a function of~$n \ge N-1$, which allows us to insert the estimate
\begin{gather*}
  \sum_{n = N}^\infty
  n^{a+b} \exp(-2 \pi n y)
\;\le\;
  \int_N^\infty
  \exp(- 2 \pi n y) n^{a+b+1}
  \frac{\rmd\!n}{n}
\le
  \frac{\Ga(a+b+1, 2 \pi (N-1) y)}{(2 \pi y)^{a+b+1}}
\tx{.}
\end{gather*}
\end{proof}

\subsection{Fourier series expansions of generalized second order Eisenstein series}

We next present the Fourier series expansions of generalized second order Eisenstein series in~\eqref{eq:def:second_order_eisenstein_series_sym_cases}. Generalized second order modular forms of type~$(\symd,\bbone)$ admit a usual Fourier series expansion, while those of type~$(\bbone,\symd)$ have a Fourier series of the form~\eqref{eq:symd_fourier_series}.

We are not able to give the Fourier series expansion in Theorem~\ref{thm:fourier_expansion_second_order_eisenstein_series} in closes form, since they feature values of a cocycle~$\phi$. We preserve the sum over double cosets in~$\Ga_\infty \backslash \SL{2}(\ZZ) \slash \Ga_\infty$ as an outer summation to emphasize the way~$\phi$ contributes to the expression. Recall that we write~$\phi(\ga)_j$ for the~$j$\thdash{} coefficient of the polynomial~$\phi(\ga)$.

\begin{theorem}
\label{thm:fourier_expansion_second_order_eisenstein_series}
Fix even integers~$\sfd \ge 0$ and~$k \ge 5 + \sfd$, and an integer~$0 \le j \le \sfd$. Then for a cocycle $\phi \in \Zpara(\SL2(\ZZ), \symd(X))$, we have
\begin{gather*}
  E_k^{[1]}(\tau;\, \phi)
\;=\;
  \sum_{r = 0}^{\sfd}
  (X - \tau)^r\,
  \sum_{n = 1}^\infty c(n)_r e(n \tau)
\end{gather*}
with Fourier coefficients
\begin{gather}
\label{eq:thm:fourier_expansion_second_order_eisenstein_series:Ephi}
  c(n)_r
\;=\;
  \sum_{\substack{\gamma \in \Ga_\infty \backslash \SL{2}(\ZZ) \slash \Ga_\infty\\\ga \not\in \Ga_\infty}}
  \frac{e(n \frac{d}{c})}{c^k}\,
  \sum_{j = r}^{\sfd}
  \phi(\gamma^{-1})_{j}\;
  \sum_{l = r}^j
  \frac{(-2 \pi i)^{k - j + l}}{(k - j + l - 1)!}
  \binom{l}{r}
  \binom{j}{l}\,
  \big(-\tfrac{d}{c}\big)^{l-r}
  n^{k - j + l - 1}
\tx{,}
\end{gather}
where~$c$ and~$d$ are the bottom entries of the double coset representative~$\ga$.

Similarly for~$\phi^\vee \in \Zpara(\SL2(\ZZ), \symd(X)^\vee)$, we have
\begin{gather*}
  E_k^{[1]}(\tau; \phi^\vee,j)
=
  \sum_{n = 1}^\infty c(n) e(n \tau)
\end{gather*}
with Fourier coefficients
\begin{gather}
\label{eq:thm:fourier_expansion_second_order_eisenstein_series:Ephij}
  c(n)
=
  \sum_{\substack{\gamma \in \Ga_\infty \backslash \SL{2}(\ZZ) \slash \Ga_\infty\\\ga \not\in \Ga_\infty}}
  \frac{(-1)^{j-k} e(n \tfrac{d}{c})}{c^{k+2j-\sfd}}
  \sum_{r = 0}^j
  \binom{j}{r}
  \frac{(2 \pi i)^{k + r}}{(k + r - 1)!}\,
  \phi^\vee(\gamma^{-1}) \Big( \big( X + \tfrac{d}{c} \big)^{\sfd-j+r} \Big)
  n^{k + r - 1}
\tx{.}
\end{gather}
\end{theorem}

\begin{proof}
Recall that the series defining~$E_k^{[1]}(\tau;\phi)$ is absolutely convergent by Proposition~5.4 of~\cite{mertens-raum-2021}. As usual in Eisenstein series calculations, we split the summation into two contributions, one from double cosets and one from~$\Ga_\infty$, and then employ the cocycle relation for~$\phi$ and~$\phi(T) = 0$ to find that:
\begin{gather*}
  E_k^{[1]}(\tau;\phi)
=
  \sum_{\ga \in \Ga_\infty \backslash \SL2(\ZZ) \slash \Ga_\infty}
  \sum_{m \in \ZZ}
  \phi\big( (\gamma T^m)^{-1} \big) \big|_k \gamma \delta
\;=\;
  \sum_{\substack{\ga \in \Ga_\infty \backslash \SL2(\ZZ) \slash \Ga_\infty\\\ga \not\in \Ga_\infty}}
  \frac{1}{c^k}
  \sum_{m \in \ZZ}
  \frac{\symd(T^{-m}) \phi(\gamma^{-1})}{(\tau + \frac{d}{c} + m)^k}
\tx{.}
\end{gather*}
We expand~$\phi(\ga^{-1})$ and insert the action of the symmetric power representation:
\begin{gather*}
  \symd(T^{-m}) \sum_{j=0}^\sfd \phi(\ga^{-1})_j X^j
\;=\;
  \sum_{j=0}^\sfd \phi(\ga^{-1})_j (X + m)^j
\tx{.}
\end{gather*}
This leads to the following sum, which we can evaluate using the Lipschitz summation formula:
\begin{gather}
\label{eq:lipschitz_factorization_summation_formula}
  \sum_{m \in \ZZ}
  \frac{(X + m)^j}{(\tau + m)^k}
\;=\;
  \sum_{m \in \ZZ}
  \frac{\big( (X - \tau) + (\tau + m) \big)^j}{(\tau + m)^k}
\;=\;
  \sum_{l = 0}^j
  (X - \tau)^l\,
  \binom{j}{l}
  \frac{(-2 \pi i)^{k-j+l}}{(k-j+l-1)!}
  \sum_{n = 1}^\infty
  n^{k-j+l-1} e(n \tau)
\tx{.}
\end{gather}
We insert~$\tau + \frac{d}{c}$ for~$\tau$, expand the result in terms of~$(X - \tau)^r$, and interchange the summations over~$j$, $l$, and~$r$ to obtain the desired formula.

We proceed in a similar way in the second case. We use continuity of~$\phi^\vee(\ga^{-1})$ to find that
\begin{gather*}
  E_k^{[1]}(\tau;\phi^\vee,j)
\;=\;
  \sum_{\substack{\ga \in \Ga_\infty \backslash \SL2(\ZZ) \slash \Ga_\infty\\\ga \not\in \Ga_\infty}}
  \frac{1}{c^k}
  \phi^\vee(\ga^{-1})\Bigg( \symd(\ga^{-1})
  \sum_{m \in \ZZ}
  \frac{(X - \ga T^m \tau)^j}{(\tau + \frac{d}{c} + m)^k}
  \Bigg)
\tx{.}
\end{gather*}
Since~$c \ne 0$ for~$\ga \not\in \Ga_\infty$, we can insert
\begin{gather*}
  \gamma T^m \tau
\;=\;
  \frac{a}{c}
  -
  \frac{1}{c^2 (\tau + \frac{d}{c} + m)}
\end{gather*}
into~$X - \ga T^m \tau$. We next expand
\begin{gather*}
  (X - \ga T^m \tau)^j
\;=\;
  \sum_{r = 0}^j
  \mbinom{j}{r}
  \big( X - \tfrac{a}{c} \bigr)^{j - r} c^{-2r} (\tau + \tfrac{d}{c} + m)^{-r}
\tx{.}
\end{gather*}
When inserting these simplification, we find that
\begin{gather*}
  E_k^{[1]}(\tau;\phi^\vee,j)
\;=\;
  \sum_{\substack{\ga \in \Ga_\infty \backslash \SL2(\ZZ) \slash \Ga_\infty\\\ga \not\in \Ga_\infty}}
  \sum_{r = 0}^j
  \mbinom{j}{r}
  \frac{1}{c^{k+2r}}
  \phi^\vee(\ga^{-1})\Bigg(
  \symd(\ga^{-1})
  \big( X - \tfrac{a}{c} \bigr)^{j - r}
  \Bigg)
  \sum_{m \in \ZZ}
  \frac{1}{(\tau + \frac{d}{c} + m)^{k+r}}
\tx{.}
\end{gather*}

The resulting expression for the Eisenstein series features
\begin{gather*}
  \symd(\ga^{-1}) \big( X - \tfrac{a}{c} \big)^{j-r}
\;=\;
  (c X + d)^{\sfd-j+r} \big( a X + b - \tfrac{a}{c} (c X + d) \big)^{j-r}
\;=\;
  (-1)^{j-r} c^{\sfd-2j+2r} \big(X + \tfrac{d}{c} \big)^{d-j+r}
\tx{.}
\end{gather*}
After inserting the right hand side, we complete the proof by utilizing the Lipschitz summation formula in~\eqref{eq:lipschitz_factorization_summation_formula} to simplify
\begin{gather*}
  E_k^{[1]}(\tau;\phi^\vee,j)
\;=\;
  \sum_{\substack{\ga \in \Ga_\infty \backslash \SL2(\ZZ) \slash \Ga_\infty\\\ga \not\in \Ga_\infty}}
  \sum_{r = 0}^j
  \mbinom{j}{r}
  \frac{(-1)^{j-r}}{c^{k+2j-\sfd}}
  \phi^\vee(\ga^{-1})\Big( \big(X + \tfrac{d}{c} \big)^{d-j+r} \Big)
  \sum_{m \in \ZZ}
  \frac{1}{(\tau + \frac{d}{c} + m)^{k+r}}
\tx{.}
\end{gather*}
\end{proof}

\subsection{Evaluation of generalized second order Eisenstein series}
\label{ssec:eval_g2es}

As in the case of the usual Eisenstein series, we want to employ the Fourier series expansion of~$E_k^{[1]}(\cdotop; \phi, j)$ to evaluate it at~$\tau \in \HS$. This requires two types of estimates, since the expression in Theorem~\ref{thm:fourier_expansion_second_order_eisenstein_series} features the infinite sum over~$\Ga_\infty \backslash \SL{2}(\ZZ) \slash \Ga_\infty$.

Our first two lemmas provide a bound on cocycle values that appear in Theorem~\ref{thm:fourier_expansion_second_order_eisenstein_series}. Convexity bounds for additive twists of~$\rmL$-functions with implicit constants were previously used, for instance, by Diamantis~\cite{diamantis-2020-preprint} to establish convergence of generalized second order Poincar\'e series.

\begin{lemma}%
\label{la:twisted_Lfunction_estimate}
Let~$f \in \rmM_k$ be a normalized Hecke eigenform of weight~$k > 3$. Then for integers~$c \ne 0$ and~$d$, and~$s \in \CC$, $\Re(s) \ge 1$, we have
\begin{gather*}
  \big| \rmL\big( f, \mfrac{-d}{c}; s \big) \big|
\;\le\;
  c^{k-2} \frac{\Ga(k-1)}{(2\pi)^{k-1}} \zeta\big( \mfrac{k-1}{2} \big)^2
\tx{.}
\end{gather*}
\end{lemma}
\begin{proof}
We use the functional equation for~$\rmL(f, -\frac{d}{c}; s)$ that is given in Theorem~3.1 of~\cite{diamantis-hoffstein-kiral-lee-2020} with~$q = 1$, $\xi = \bbone$, $r = 1$, $M = M_1 = c$, $\alpha = a= -d$, and~$u = 0$. Notice that the statement of that theorem features a contribution~$a(r \slash n e)$ which must be read as a Fourier coefficient of~$f$ as opposed to the product of~$a$ with~$r \slash n e$. We have
\begin{gather*}
  \Lambda\big( f, \mfrac{-d}{c}; s \big)
=
  i^k c^{k-2s}
  \Lambda\big( f, \mfrac{-d}{c}; k-s \big)
\tx{,}
\end{gather*}
where~$\Lambda( f, -d / c\mspace{1mu}; s ) = \Ga(s) (2 \pi)^{-s}\, \rmL( f, -\frac{d}{c}; s )$ is the completed~$\rmL$-function.

Since~$k > 3$ and~$f$ is a normalized Hecke eigenform, we have
\begin{gather*}
  \big| \rmL\big( f, \mfrac{-d}{c}; k-s \big) \big|
\le
  \sum_n \big| e\big(n \mfrac{-d}{c}\big) c(f;n) n^{s-k} \big|
=
  \sum_n |c(f;n)| n^{1-k}
\le
  \sum_n \sigma_0(n) n^{\frac{1-k}{2}}
=
  \zeta\big(\mfrac{k-1}{2}\big)^2
\tx{.}
\end{gather*}
We insert this into the right hand side of the functional equation and then apply the Phragm\'en-Lindel\"of principle to finish the proof.
\end{proof}

\begin{lemma}
\label{la:cocycle_estimate}
Let~$f \in \rmM_k$ be a normalized Hecke eigenform of weight~$k > 3$, and recall the cocycle~$\phi_{\cE(f)}$ associated in~\eqref{eq:def:cocycle_modular_deficit} to the Eichler integral in~\eqref{eq:def:eichler_integral}. Then for~$\ga = \begin{psmatrix} a & b \\ c & d \end{psmatrix} \in \SL{2}(\ZZ)$ with~$|d| < |c|$, we have
\begin{gather}
\label{eq:la:cocycle_estimate}
  \sum_{j = 0}^{k-2} \big| \phi_{\cE(f)}(\ga^{-1})_j \big|
\le
  c^{k-2}\, \frac{e^2\, \Ga(k-1)^2}{(2 \pi)^k}\,
  \zeta\big(\mfrac{k-1}{2}\big)^2
\tx{.}
\end{gather}
\end{lemma}
\begin{proof}
We combine the explicit expression~\eqref{eq:cocycle_explicit_expression} for~$\phi_{\cE(f)}(\ga^{-1})$ with the estimate in Lemma~\ref{la:twisted_Lfunction_estimate} to find that the left hand side of~\eqref{eq:la:cocycle_estimate} is bounded by
\begin{align*}
&
  \sum_{j = 0}^{k-2}
  \frac{(k-2)!}{(k-2-j)!}
  \sum_{r = 0}^j
  \frac{| d \slash c |^{j-r}}{(j-r)!\, (2 \pi)^{r+1}}\,
  c^{k-2} \frac{\Ga(k-1)}{(2\pi)^{k-1}} \zeta\big( \mfrac{k-1}{2} \big)^2
\\
\le\;&
  c^{k-2} \frac{\Ga(k-1)}{(2\pi)^{k}} \zeta\big( \mfrac{k-1}{2} \big)^2\,
  \sum_{j = 0}^{k-2}
  \sum_{r = 0}^j
  \frac{(k-2)!}{(k-2-j)!(j-r)!}
\tx{.}
\end{align*}
The double sum is bounded by~$e^2 \Ga(k-1)$, which establishes the result.
\end{proof}

The next lemma estimates the tail of the expression for~$c(n)$ in Theorem~\ref{thm:fourier_expansion_second_order_eisenstein_series}, based on the bound in the previous lemma.

\begin{lemma}
\label{la:fourier_coefficient_tail_estimate}
Let~$f \in \rmM_l$ be a normalized Hecke eigenform of weight~$l > 3$ and set~$\sfd = l-2$. Fix an integer~$0 \le l \le \sfd$ and an even integer~$k > 2 + 2(\sfd-j)$. Then given a positive integer~$C$, we have
\begin{gather}
\label{eq:la:fourier_coefficient_tail_estimate}
\begin{aligned}
&
  \Bigg|
  \sum_{\substack{\gamma \in \Ga_\infty \backslash \SL{2}(\ZZ) \slash \Ga_\infty\\|c| \ge C}}
  \frac{(-1)^{j-k} e(n \tfrac{d}{c})}{c^{k+2j-\sfd}}
  \sum_{r = 0}^j
  \mbinom{j}{r}
  \frac{(2 \pi i)^{k + r}}{(k + r - 1)!}\,
  \phi_{\cE(f)}^\vee(\gamma^{-1}) \Big( \big( X + \tfrac{d}{c} \big)^{\sfd-j+r} \Big)
  n^{k + r - 1}
  \Bigg|
\\
\le\;&
  n^{k + j - 1}\,
  \Big( \bbone_{C = 1} + \frac{1}{(k+2j-2l+2)C^{k+2j-2l+3}} \Big)\,
  \frac{2^j e^2\, (l-2)!^2\, (2 \pi)^{k + j - l}}{(k - 1)! }\,
  \zeta\big(\mfrac{l-1}{2}\big)^2
\tx{,}
\end{aligned}
\end{gather}
where~$c$ and~$d$ on the left hand side are the bottom entries of~$\ga$, and~$\bbone_{C = 1}$ equals one if and only if\/~$C = 1$ and otherwise zero.
\end{lemma}

\begin{remark}
A better bound for the tails of~\eqref{eq:thm:fourier_expansion_second_order_eisenstein_series:Ephij} can be achieved by intertwining the summation over~$n$ and~$\ga$ to utilize Lambert series as in Remark~\ref{rm:la:errorbound_ekd_partial_lambert_series}.
\end{remark}

\begin{proof}[Proof of Lemma~\ref{la:fourier_coefficient_tail_estimate}]
For simplicity, we write~$\phi^\vee$ for the cocycle in~\eqref{eq:la:fourier_coefficient_tail_estimate}. The sum over double coset can be written as a sum over integers~$c \ge C$ and~$d \pmod{c}$. We can and will assume that~$|d| < c$ as in the assumptions of Lemma~\ref{la:cocycle_estimate}. We expand the power of~$X + \frac{d}{c}$ and apply the coefficient formula for~$\phi^\vee$ in~\eqref{eq:def:dual_cocycle} to find that
\begin{align*}
  \Big|
  \phi^\vee(\ga^{-1})\Big(
  (X + \tfrac{d}{c})^{\sfd-j+r}
  \Big)
  \Big|
\;=\;
  \Big|
  \sum_{i = 0}^\sfd
  (-1)^{\sfd-i} \mbinom{\sfd}{i}^{-1}
  \phi(\ga^{-1})_{\sfd-i}\;
  \mbinom{\sfd-j+r}{i}
  \big(\tfrac{d}{c}\big)^{\sfd-j+r-i}
  \Big|
\;\le\;
  \sum_{i = 0}^\sfd
  \big| \phi(\ga^{-1})_{\sfd-i} \big|
\tx{.}
\end{align*}

We insert this into the left hand side of~\eqref{eq:la:fourier_coefficient_tail_estimate}, which is thus bounded by
\begin{align*}
&
  \sum_{c = C+1}^\infty
  \sum_{d = 0}^{c-1}
  \frac{1}{c^{k+2j-l+2}}
  \sum_{r = 0}^j
  \mbinom{j}{r}
  \frac{(2 \pi)^{k + j}}{(k - 1)!}\,
  c^{l-2}\, \frac{e^2\, \Ga(l-1)^2}{(2 \pi)^l}\,
  \zeta\big(\mfrac{l-1}{2}\big)^2
  n^{k + r - 1}
\\
\le&
  n^{k + j - 1}\,
  \frac{2^j e^2, \Ga(l-1)^2\, (2 \pi)^{k + j - l}}{(k - 1)!}\,
  \zeta\big(\mfrac{l-1}{2}\big)^2
  \sum_{c = C}^\infty
  \frac{1}{c^{k+2j-2l+4}}
\end{align*}
We finish the proof by estimating the sum over~$c$ by the integral from~$C-1$ or~$C$ to~$\infty$ of~$c^{-k-2j+2l-4}$.
\end{proof}

Our final lemma concerns the tails of the Fourier series expansion of~$E_k(\cdotop; \phi^\vee, j)$. 
\begin{lemma}
\label{la:fourier_expansion_tail_estimate}
Let~$f \in \rmM_l$ be a normalized Hecke eigenform of weight~$l > 3$ and set~$\sfd = l-2$. Fix an integer~$0 \le l \le \sfd$ and an even integer~$k > 2 + 2(\sfd-j)$. Let~$c(n)$ be the Fourier coefficients of~$E_k(\cdotop; \phi^\vee_{\cE(f)}, j)$ given in~\eqref{eq:thm:fourier_expansion_second_order_eisenstein_series:Ephij}. Given a positive integer~$N$, we have
\begin{gather}
\label{eq:la:fourier_expansion_tail_estimate}
  \Bigg|
  \sum_{n = N}^\infty
  c(n) e(n \tau)
  \Bigg|
\;\le\;
  \frac{k+2j-2l+3}{k+2j-2l+2}
  \frac{2^j e^2\, (l-2)!^2}{(k - 1)! (2 \pi)^l}\,
  \zeta\big(\mfrac{l-1}{2}\big)^2
  \frac{\Ga(k+j, 2 \pi (N-1) y)}{y^{k+j}}
\tx{.}
\end{gather}
\end{lemma}
\begin{proof}
Lemma~\ref{la:fourier_coefficient_tail_estimate} with~$C = 1$ provides an estimate for the right hand side of~\eqref{eq:thm:fourier_expansion_second_order_eisenstein_series:Ephij}. The left hand side of~\eqref{eq:la:fourier_expansion_tail_estimate} is thus bounded by
\begin{align*}
&
  \frac{k+2j-2l+3}{k+2j-2l+2}
  \frac{2^j e^2\, (l-2)!^2\, (2 \pi)^{k + j - l}}{(k - 1)! }\,
  \zeta\big(\mfrac{l-1}{2}\big)^2
  \sum_{n=N}^\infty n^{k + j - 1} \exp(-2 \pi n y)
\\
\le\;&
  \frac{k+2j-2l+3}{k+2j-2l+2}
  \frac{2^j e^2\, (l-2)!^2\, (2 \pi)^{k + j - l}}{(k - 1)! }\,
  \zeta\big(\mfrac{l-1}{2}\big)^2
  \frac{\Ga(k+j, 2 \pi (N-1) y)}{(2 \pi y)^{k+j}}
\tx{.}
\end{align*}
\end{proof}

\section{Saturation of generalized second order modular forms}
\label{sec:saturation}

In this section, we will examine the saturation at the Ramanujan $\Delta$\nbd function of some~$\rmM_\bullet$\nbd ideals generated by Eisenstein series in~$\rmM_{\bullet,\sfd}$. Recall that given a commutative ring~$R$, two $R$\nbd modules~$N \subseteq M$, and an element $r\in R$, the saturation of~$N$ in~$M$ at~$r$ is defined as
\begin{gather*}
  ( N : r^\infty )
\;:=\;
  \big\{
  m \in M \,:\,
  \exists n \in \ZZ_{\ge 0} :\, r^n m \in N
  \big\}
\tx{.}
\end{gather*}

Given integers~$\sfd \ge 0$ and~$k_0 > 2 + \sfd$, we consider the~$\rmM_\bullet$\nbd module
\begin{gather}
\label{eq:def:eisenstein_module_symd}
  \rmE_{\ge k_0,\sfd}
\;:=\;
  \lspan\, \rmM_\bullet\,
  \big\{
  E_k(\,\cdot\,; \sfd, j) \,:\,
  k \ge k_0,\,
  0 \le j \le \sfd
  \big\}
\;\subseteq\;
  \rmM_{\bullet,\sfd}
\tx{.}
\end{gather}
If~$\sfd = 0$, we suppress it from our notation and write~$\rmE_{\ge k_0}$. In our first two propositions, we describe the saturation of~\eqref{eq:def:eisenstein_module_symd}.

\begin{proposition}
\label{prop:saturation_at_delta_scalar_valued}
Given an integer~$k_0 > 2$, we have
\begin{gather*}
  \big( \rmE_{\ge k_0} : \Delta^\infty \big)
\;=\;
  \rmM_\bullet
\tx{.}
\end{gather*}
\end{proposition}
\begin{proof}
Since~$\rmM_\bullet$ is generated as an~$\rmM_\bullet$\nbd module by~$1$, it suffices to show that there is a non-negative integer~$m$ such that~$\Delta^m \in \rmE_{\ge k_0}$. In other words, it suffices to show that for some integer~$m$ with~$12 m \ge 2 k_0$, we have
\begin{gather*}
  \Delta^m
\in
  \CC E_{12m}
\,+\,
  \lspan \CC
  \big\{
  E_{12m - k} E_k \,:\,
  k_0 \le k \le 12m - 4,
  \big\}
\;\subseteq\;
  \rmE_{\ge k_0}
\tx{.}
\end{gather*}
This follows by a standard argument along the lines of Kohnen--Zagier~\cite{rankin-1952,kohnen-zagier-1984}. See also~\cite{kohnen-martin-2008}.
\end{proof}

\begin{proposition}
\label{prop:saturation_at_delta_vector_valued}
Given integers~$0 \le \sfd$ and~$k_0\geq 5 + \sfd$, we have
\begin{gather*}
  \big( \rmE_{\ge k_0,\sfd} : \Delta^\infty \big)
  \;=\;
  \rmM_{\bullet,\sfd}
\tx{.}
\end{gather*}
\end{proposition}
\begin{proof}
We establish the statement by showing via induction on~$0 \le j \le \sfd$ that
\begin{gather}
\label{eq:prop:saturation_at_delta_vector_valued:induction}
  \big( \rmE_{\ge k_0,\sfd} : \Delta^\infty \big) \cap \rmM_{\bullet,\sfd}[j] 
\;=\;
  \big( \rmE_{\ge k_0,\sfd} \cap \rmM_{\bullet,\sfd}[j] : \Delta^\infty \big)
\;=\;
  \rmM_{\bullet,\sfd}[j]
\tx{.}
\end{gather}

If~$j = \sfd$, then the projection~$\pilow_{k,\sfd,\sfd}$ in~\eqref{eq:la:projection_to_lowest_symd_component} is an isomorphism. By Corollary~\ref{cor:eisenstein_series_symd_filtration_steps} it maps Eisenstein series to Eisenstein series. Thus~\eqref{eq:prop:saturation_at_delta_vector_valued:induction} follows directly from Proposition~\ref{prop:saturation_at_delta_scalar_valued}.

Assume now that~$j < \sfd$ and that~\eqref{eq:prop:saturation_at_delta_vector_valued:induction} is true for~$j+1$ instead of~$j$. The projection~$\pilow_{k,\sfd,j}$ and Corollary~\ref{cor:eisenstein_series_symd_filtration_steps} yield the following diagram:
\begin{center}
\begin{tikzpicture}
\matrix(m)[matrix of math nodes,
row sep = 2em, column sep = 3em,
text height = 1.5em, text depth = 1.5ex]
{%
  \big( \rmE_{\ge k_0,\sfd} : \Delta^\infty \big)
  \cap \rmM_{\bullet,\sfd}[j+1] 
& \big( \rmE_{\ge k_0,\sfd} : \Delta^\infty \big)
  \cap \rmM_{\bullet,\sfd}[j] 
& \big( \rmE_{\ge k_0+2j-\sfd} : \Delta^\infty \big)
\\ 
  \rmM_{\bullet,\sfd}[j+1] 
& \rmM_{\bullet,\sfd}[j] 
& \rmM_{\bullet}
\\};

\path[right hook-latex]
(m-1-1) edge (m-1-2)
(m-2-1) edge (m-2-2);
\path[-latex]
(m-1-2) edge (m-1-3)
(m-2-2) edge (m-2-3);
\path[left hook-latex]
(m-1-1) edge (m-2-1)
(m-1-2) edge (m-2-2)
(m-1-3) edge (m-2-3);
\end{tikzpicture}
\end{center}
By induction and by Proposition~\ref{prop:saturation_at_delta_scalar_valued}, the left and right vertical arrows are equalities. To finish the proof by the Four-Lemma, we have to show that the co-kernel of the top right horizontal map injects into the co-kernel of the bottom right one.

To prove this, we observe that there is a natural splitting~$\sigma$ to the projection from~$E_{\ge k_0,\sfd}$ to~$E_{\ge k_0 + 2j - \sfd}$. Given a modular form~$f$ and an integer~$h \ge 0$ with~$\Delta^h f \in \rmE_{\ge k_0+2j-\sfd}$, the obstruction to finding a preimage of~$f_j$ in the top center space is given by the first~$h$ Fourier coefficients of
\begin{gather*}
  \sigma(\Delta^h f)
  \,+\,
 \big( \rmE_{\ge k_0,\sfd} : \Delta^\infty \big) \cap \rmM_{\bullet,\sfd}[j+1] 
\;=\;
  \sigma(\Delta^h f)
  \,+\,
 \rmM_{\bullet,\sfd}[j+1]
\tx{.}
\end{gather*}
The right hand side coincides with the obstruction that arises in the bottom row, and we hence conclude the proof.
\end{proof}

To make our statement about generalized second order modular forms more consice, we introduce the
following spaces of generalized second order Eisenstein series, analogous to
$\rmE_{\geq k_0}$ and $\rmE_{\geq k_0,\sfd}$. For integers~$\sfd \ge 0$
and~$k_0 \geq 5 + \sfd$, we set
\begin{gather}
\label{eq:def:generalized_second_order_eisenstein_modules}
\begin{aligned}
  \rmE^{\mathrm{pure}\,[1]}_{\ge k_0} \big( \symd(X), \bbone \big)
\;&{}:=\;
  \lspan \rmM_\bullet \big\{
  E^{[1]}_k\big( \,\cdot\,;\phi \big)\,:\,
  k \ge k_0,\,
  0 \le j \le \sfd,\,
  \phi \in \Zpara\big( \SL{2}(\ZZ), \symd(X) \big)
  \big\}
\tx{,}
\\
  \rmE^{\mathrm{pure}\,[1]}_{\ge k_0}\big( \bbone, \symd(X) \big)
\;&{}:=\;
  \lspan \rmM_\bullet \big\{
  E^{[1]}_k\big( \,\cdot\,;\phi, j\big)\,:\,
  k \ge k_0,\,
  0 \le j \le \sfd,\,
  \phi \in \Zpara\big( \SL{2}(\ZZ), \symd(X)^\vee \big)
  \big\}
\tx{.}
\end{aligned}
\end{gather}

For the proof of our next theorem, we leverage
Proposition~\ref{prop:saturation_at_delta_scalar_valued} and
Proposition~\ref{prop:saturation_at_delta_vector_valued} in the same way as we employed
Proposition~\ref{prop:saturation_at_delta_scalar_valued} in the proof of
Proposition~\ref{prop:saturation_at_delta_vector_valued}.
\begin{theorem}
\label{thm:snd_order_eisenstein_saturation}
Given integers~$0 \le \sfd$ and~$k_0 \geq 5 + \sfd$, we have
\begin{align*}
  \big(
  \rmE^{\mathrm{pure}\,[1]}_{\ge k_0}(\symd(X), \bbone) + \rmM_{\bullet,\sfd}
  \,:\, \Delta^\infty\big)
\;&{}=\;
  \rmM^{[1]}_\bullet(\symd(X), \bbone)
\quad\tx{and}\\
  \big(
  \rmE^{\mathrm{pure}\,[1]}_{\ge k_0}(\bbone, \symd(X)) + \rmM_\bullet
  \,:\, \Delta^\infty
  \big)
\;&{}=\;
  \rmM^{[1]}_\bullet(\bbone, \symd(X))
\tx{.}
\end{align*}
\end{theorem}

\begin{remark}
\label{rm:thm:snd_order_eisenstein_saturation}
The proof of Theorem~\ref{thm:snd_order_eisenstein_saturation} can be translated to the language of vector-valued modular forms~\cite{mertens-raum-2021}, and then yields an induction over the socle series. In particular, when combined with~\cite{raum-xia-2020} it can be generalized to all modular forms of virtually real-arithmetic type. In this form it subsumes all iterated Eichler-Shimura integrals.
\end{remark}

\begin{proof}
We focus on the second case, and leave the first one, whose proof proceeds analogously, to the reader. Observe that when viewing~$\rmM_{\bullet,\sfd}$ as a trivial~$\SL{2}(\ZZ)$ representation, we have an identification
\begin{gather}
\label{eq:thm:snd_order_eisenstein_saturation:cocycle}
  \Zpara\big( \SL{2}(\ZZ),\, \symd(X)^\vee \otimes \rmM_{k,\sfd} \big)
\;=\;
  \Zpara\big( \SL{2}(\ZZ),\, \symd(X)^\vee \big) \otimes \rmM_{k,\sfd}
\tx{.}
\end{gather}
Definition~\ref{def:generalized_second_order_modular_forms} yields a map from~$\rmM_k^{[1]}(\bbone, \symd(X))$ to~\eqref{eq:thm:snd_order_eisenstein_saturation:cocycle}. From the definition of generalized second order Eisenstein series in~\eqref{eq:def:eisenstein_series_symd}, we infer that it sends~$E_k(\,\cdot\,;\phi,j)$ to~$E_k(\,\cdot\,; \sfd,j)$.

We assemble this into a commutative diagram akin to the one that we used in the proof of Proposition~\ref{prop:saturation_at_delta_vector_valued}.
\begin{center}
\begin{tikzpicture}
\matrix(m)[matrix of math nodes,
row sep = 2em, column sep = 3em,
text height = 1.5em, text depth = 1.5ex]
{%
  \rmM_\bullet
& \big(
  \rmE^{\mathrm{pure}\,[1]}_{\ge k_0}(\bbone, \symd(X)) + \rmM_\bullet
  \,:\, \Delta^\infty
  \big)
& \Zpara\big( \SL{2}(\ZZ),\, \symd(X)^\vee \big)
  \otimes 
  \big( \rmE_{\ge k_0, \sfd} : \Delta^\infty \big)
\\
  \rmM_\bullet
& \rmM_\bullet\big( \bbone, \symd(X) \big)
& \Zpara\big( \SL{2}(\ZZ),\, \symd(X)^\vee \big)
  \otimes \rmM_{\bullet,\sfd}
\\};

\path[right hook-latex]
(m-1-1) edge (m-1-2)
(m-2-1) edge (m-2-2);
\path[-latex]
(m-1-2) edge (m-1-3)
(m-2-2) edge (m-2-3);
\path[right hook-latex]
(m-1-1) edge (m-2-1)
(m-1-2) edge (m-2-2)
(m-1-3) edge (m-2-3);
\end{tikzpicture}
\end{center}
For clarity, note that the bottom left space arises from~$\rmM_\bullet$ in the statement of Theorem~\ref{thm:snd_order_eisenstein_saturation} and could be replaced by~$\rmE_{\ge k_0}$ without altering this proof. Now the left vertical arrow is trivially an equality, and the right one is so by Proposition~\ref{prop:saturation_at_delta_vector_valued}. We finish the proof by applying the Four-Lemma after inspecting the co-kernels of the right horizontal maps exactly as in the proof of Proposition~\ref{prop:saturation_at_delta_vector_valued}.
\end{proof}


\section{The Eichler integral of the discriminant function}
\label{sec:bootstrap}

In this section we demonstrate and compare two ways to recover the Eichler integral~$\cE(\Delta)$ of the Ramanujan discriminant function via Theorem~\ref{thm:snd_order_eisenstein_saturation}. Calculations were performed using the Nemo computer algebra package~\cite{fieker-hart-hofmann-johansson-2017}. Recall from Section~\ref{ssec:eich_int_g2mf} that we have
\begin{gather*}
  \cE(\Delta)
\in
  \rmM^{[1]}_{-10}\big( \bbone, \sym^{10} \big)
\tx{.}
\end{gather*}
The mechanics of our example apply to any element of~$\rmM^{[1]}_k( \bbone, \symd )$. One can derive from Sections~\ref{sec:eis} and~\ref{sec:saturation} a similar procedure for generalized second order modular forms of type~$(\symd, \bbone)$. Observe that in our specific case, we have a precise expression for the cusp expansion:
\begin{gather*}
  \cE(\Delta)(\tau)
=
  \frac{- \Gamma(k-1)}{(2 \pi i)^{k-1}}\,
  \sum_{n=1}^\infty n^{1-k} c(\Delta;\,n) e(n \tau)
\tx{,}
\end{gather*}
where the~$c(\Delta; n)$ are the Fourier coefficients of~$\Delta$. Since~$\Delta$ is a level-$1$ modular form, this allows us to evaluate~$\cE(\Delta)$ numerically in an efficient way at any point of~$\HS$. In applications that follow the method outlined in this section such efficient evaluations are not necessarily possible, but we need to rely on the numerical values at a small number of points or alternatively some of the Fourier coefficients.

Throughout this section, we write~$\varphi$ for the cocycle associated with~$\cE(\Delta)$. Paralleling the structure of Theorem~\ref{thm:snd_order_eisenstein_saturation}, first we use~$\phi$ to express~$\Delta^h \cE(\Delta)$ for a fixed integer~$h \ge 2$ as the sum of an~a~priori undetermined modular form~$f_h$ of weight~$12h-10$ and a linear combination of products of Eisenstein series and generalized second order Eisenstein series. In a second step we determine~$f_h$ from the numerical evaluation of~$\cE(\Delta)$.

The trade-offs are as follows: for~$h = 2$, the modular form~$f_h \in \rmM_{12h-10}$ vanishes, but the generalized second order Eisenstein series that appear converge slowly. In this case, $\Delta^h \cE(\Delta)$ is determined by its cocycle~$\Delta^h \phi$. The larger~$h$ is, the faster the relevant generalized second order Eisenstein series converge, but the more evaluations of~$\cE(\Delta)$ at suitable points of~$\HS$ are needed to determine~$f_h$.

This can be ``bootstapped'', that is, we can use the slowly converging but fully determined expression for~$h = 2$ to compute a faster converging expression for a different~$h$. By repeating this procedure, we get successively faster rates of convergence. However, attention must be paid to the possible accumulation of numerical errors.

We give a concrete description of the above procedure in the special cases~$h = 2$ and~$h = 5$.

\paragraph{An expression for $\Delta^2 \cE(\Delta)$}

In order to compute values of~$\cE(\Delta)$, we assume the values of~$\Delta$ and consider
\begin{gather*}
  \Delta^2\, \cE(\Delta)(\tau)
\in
  \rmM_{14}^{[1]}\big( \bbone, \sym^{10} \big)
\tx{.}
\end{gather*}

The cocycle of~$\Delta^2 \cE(\Delta)$ can coincides with the one of~$\cE(\Delta)$. Via~\eqref{eq:prop:eichler_integral:cocycle}, we express it in terms of~$\Delta^2\, (X - \tau)^{10}$, which leads us to considering the relation
\begin{gather*}
  \Delta^2
=
  c_1
  E_{24}
+
  c_2
  E_4 E_{20}
+
  c_3
  E_6 E_{18}
\end{gather*}
with coefficients
\begin{gather*}
  c_1
=
  \mfrac{70933884092880847}{655667091203604480000}
\tx{,}\quad
  c_2
=
  -
  \mfrac{1912212380581}{26083744727040000}
\tx{,}\quad
  c_3
=
  -
  \mfrac{97825033079}{2804992903544832}
\tx{.}
\end{gather*}
It allows us to conclude that for some modular form~$f\in\rmM_{14}$ we have
\begin{gather}
\label{eq:delta_eichler_integral_as_eisenstein_product}
  \Delta^2\, \cE(\Delta)
\;=\;
  f
\,-\,
  c_1
  E_{14}^{[1]}(\,\cdot\,;\,\phi,10)
-
  c_2
  E_4 E_{10}^{[1]}(\,\cdot\,;\,\phi,10)
-
  c_3
  E_6 E_{8}^{[1]}(\,\cdot\,;\,\phi,10)
\tx{.}
\end{gather}

The Fourier series expansions of generalized second Eisenstein series in
Theorem~\ref{thm:fourier_expansion_second_order_eisenstein_series} reveals that~$f$ is
a cusp form, which implies that~$f = 0$, and thus $\Delta^2\cE(\Delta)$ is fully
determined by~\eqref{eq:delta_eichler_integral_as_eisenstein_product}.

\paragraph{An expression for $\Delta^5 \cE(\Delta)$}

Performing the same steps as before, we calculate that
\begin{gather*}
  \Delta^5
\;=\;
  c_1 E_{60}
+ c_2 E_4 E_{56}
+ c_3 E_6 E_{54}
+ c_4 E_8 E_{52}
+ c_5 E_{10} E_{50}
+ c_6 E_{12} E_{48}
\end{gather*}
with coefficients $c_i \in \QQ$ whose numerators and denominators are on the order
of $10^{80}$. Approximatively, we have
\begin{alignat*}{3}
  c_1 &{}\approx -2.7768336 \cdot 10^{-7}
\tx{,}\;
&
  c_2 &{}\approx 1.41222683 \cdot 10^{-7}
\tx{,}\;
&
  c_3 &{}\approx 9.5721307  \cdot 10^{-8}
\tx{,}\;
\\
  c_4 &{}\approx 3.3445439 \cdot  10^{-8}
\tx{,}
&
  c_5 &{}\approx 6.67556586 \cdot 10^{-9}
\tx{,}
&
  c_6 &{}\approx 6.1836892 \cdot 10^{-10}
\tx{.}
\end{alignat*}
Similar to the previous case, using only the cocycle associated with~$\cE(\Delta)$, we find that there exists a cusp form~$f \in \rmS_{50}$ with
\begin{multline*}
  \Delta^5 \cE(\Delta)
\;=\;
  f
- c_1 E_{50}^{[1]}(\,\cdot\,;\,\phi,10)
- c_2 E_4 E_{46}^{[1]}(\,\cdot\,;\,\phi,10)
- c_3 E_6 E_{44}^{[1]}(\,\cdot\,;\,\phi,10)
- c_4 E_8 E_{42}^{[1]}(\,\cdot\,;\,\phi,10)\\
- c_5 E_{10} E_{40}^{[1]}(\,\cdot\,;\,\phi,10)
- c_6 E_{12} E_{38}^{[1]}(\,\cdot\,;\,\phi,10)
\tx{.}
\end{multline*}
For simplicity, we write~$f_E$ for the Eisenstein terms on the right hand side, so that the previous relation can be written as~$\Delta^5 \cE(\Delta) = f - f_E$.

Since~$\dim(\rmS_{50}) = 3$, it suffices to evaluate~$\Delta^5 \cE(\Delta) + f_E$ at three suitable points on the Poincar\'e upper half plane and compare with the evaluation of an appropriate basis of~$\rmS_{50}$, to determine $f$. This is where the evaluation of~$\Delta^5 \cE(\Delta)$ is needed, which can be obtained via the case~$h = 2$, as explained before. In the case at hand, in which we merely illustrate the application of Theorem~\ref{thm:snd_order_eisenstein_saturation}, we evaluate~$\cE(\Delta)$ directly via its Fourier expansion in this step. For~$\rmS_{50}$, we use a basis~$f_1, f_2, f_3$ of~$\rmS_{50}$ with initial Fourier expansions
\begin{gather*}
  f_1 = e(\tau) - 2064 e(2 \tau) + 1358532 e(3 \tau) + \cO\big(e(4 \tau)\big)
\tx{,}\\
  f_2 = e(2\tau) - 1080 e(3\tau) + \cO\big(e(4 \tau)\big)
\tx{,}\quad
  f_3 = e(3\tau)q^3 + \cO\big(e(4 \tau)\big)
\tx{.}
\end{gather*}
As points on the upper half plane, we use
\begin{gather*}
  \tau_1 = 1 + 2 i
\tx{,}\quad
  \tau_2 = \mfrac{1}{2} + 3 i
\tx{,}\quad
  \tau_3 = \mfrac{1}{3} + 4 i
\tx{,}
\end{gather*}
which yield a matrix with entries~$f_i(\tau_j)$, $1 \le i,j \le 3$, of sufficiently large determinant, allowing for numerical stability when solving for~$f$.
%
Theorem~\ref{thm:fourier_expansion_second_order_eisenstein_series} and the tail estimates in Lemmas~\ref{la:fourier_coefficient_tail_estimate} and~\ref{la:fourier_expansion_tail_estimate} yield
\begin{align*}
  \Delta^5 \cE(\Delta)(\tau_1) + f_E(\tau_1)
&{}\approx
  7.334 \cdot 10^{-46}  + i 5.603 \cdot 10^{-25}
\tx{,}\\
  \Delta^5 \cE(\Delta)(\tau_2) + f_E(\tau_2)
&{}\approx
  1.574 \cdot 10^{-51}  + i 3.940 \cdot 10^{-31}
\tx{,}\\
  \Delta^5 \cE(\Delta)(\tau_3) + f_E(\tau_3)
&{}\approx
  -9.190 \cdot 10^{-37} - i 1.919 \cdot 10^{-36}
\tx{,}
\end{align*}
from which we obtain~$f = d_1 f_1 + d_2 f_2 + d_3 f_3$ with
\begin{align*}
  d_1 &\approx -3.504686 \cdot 10^{-38} - i 1.4752509 \cdot 10^{-25}
\tx{,}\\
  d_2 &\approx -5.371557 \cdot 10^{-30} - i 2.207136 \cdot 10^{-15}
\tx{,}\\
  d_3 &\approx 1.5378812 \cdot 10^{-24} - i 1.0722708 \cdot 10^{-10}
\tx{.}
\end{align*}


\ifbool{nobiblatex}{%
  \bibliographystyle{alpha}%
  \bibliography{bibliography.bib}%
}{%
  \Needspace*{4em}
  \printbibliography[heading=none]
}


\Needspace*{3\baselineskip}
\noindent
\rule{\textwidth}{0.15em}

{\noindent\small
Chalmers tekniska högskola,
Institutionen för Matematiska vetenskaper,
SE-412 96 Göteborg, Sweden\\
E-mail: \url{albin.ahlback@gmail.com}\\
Homepage: \url{https://albinahlback.gitlab.io}
}\vspace{.5\baselineskip}

{\noindent\small
Chalmers tekniska högskola och Göteborgs Universitet,
Institutionen för Matematiska vetenskaper,
SE-412 96 Göteborg, Sweden\\
E-mail: \url{tobmag@chalmers.se}\\
Homepage: \url{https://tobiasmagnusson.com}
}\vspace{.5\baselineskip}

{\noindent\small
Chalmers tekniska högskola och G\"oteborgs Universitet,
Institutionen f\"or Matematiska vetenskaper,
SE-412 96 G\"oteborg, Sweden\\
E-mail: \url{martin@raum-brothers.eu}\\
Homepage: \url{http://martin.raum-brothers.eu}
}


\ifdraft{%
\listoftodos%
}

\end{document}
